\documentclass[a4paper,twoside,12pt]{article}
\usepackage[english]{babel}

\usepackage{amsmath,amsfonts,amssymb,amsthm}
\newtheorem{teo}{Theorem}
\newtheorem{lem}{Lemma}

\newtheorem{remark}{Remark}

\newtheorem{ex}{Example}

\newtheorem{corol}{Corollary}

 \title{{\bf Natural lacunae method and Schatten-von Neumann classes  of the convergence exponent}}

\author{Maksim \,V.~Kukushkin   \\ \\
 \small  \textit{Moscow State University of Civil Engineering, 129337,  Moscow, Russia}\\
 \small\textit{Kabardino-Balkarian Scientific Center, RAS, 360051,  Nalchik, Russia}\\
\textit{\small\textit{kukushkinmv@rambler.ru}} }

\date{}
\begin{document}

\maketitle

\begin{abstract}
The first our aim is to clarify the results obtained by Lidskii  devoted to the decomposition on the root vector system of the non-selfadjoint operator. We use a technique of the entire function theory and introduce  a so-called  Schatten-von Neumann class of the convergence  exponent. Considering strictly accretive operators satisfying special conditions formulated in terms of the norm, we construct a sequence of contours of the power type on the  contrary to the results by Lidskii, where a sequence of contours of  the exponential  type was used.

\end{abstract}
\begin{small}\textbf{Keywords:}
 Strictly accretive operator;  Abel-Lidskii basis property;   Schatten-von Neumann  class; convergence exponent; counting function.   
\end{small}

\section{Introduction}

  Generally the concept origins from the  well-known fact that the eigenvalues of the compact selfadjoint operator form a basis in the closure of its range. The question what happens in the case when the operator is non-selfadjoint is rather complicated and  deserves to be considered as a separate part of the spectral theory. Basically, the aim of the mentioned  part of the spectral theory   are   propositions on the convergence of the root vector series in one or another sense to an element belonging to the closure of the operator range. Here, we should note when we  say a sense we mean   Bari, Riesz, Abel (Abel-Lidskii) senses of the series convergence  \cite{firstab_lit:2Agranovich1994},\cite{firstab_lit:1Gohberg1965}. A reasonable question that appears is about minimal conditions that guaranty the desired result, for instance in the mentioned papers  there considered a domain of the parabolic type containing the spectrum of the operator. In the paper \cite{firstab_lit:2Agranovich1994}, non-salfadjoint operators with the special condition imposed on the numerical range of values are considered. The main advantage of this result is a comparatively weak condition imposed upon the numerical range of values comparatively with the sectorial condition (see definition of the sectorial operator). Thus, the convergence in the Abel-Lidskii sense  was established for an operator class wider when the class of sectorial operators. Here, we  make comparison between results devoted to operators with the discrete spectra and operators with the compact resolvent for they can be easily reformulated from one to the other field.

  The central idea of this paper  is to formulate sufficient conditions of the Abel-Lidskii basis property of the root functions system for a sectorial non-selfadjoint operator of the special type. Considering such an operator class we strengthen a little the condition regarding the semi-angle of the sector, but weaken a great deal conditions regarding the involved parameters. Moreover, the central aim generates  some prerequisites to consider technical peculiarities such as a newly constructed sequence of contours of the power type on the contrary to the Lidskii results \cite{firstab_lit:1Lidskii}, where a sequence of the contours of the  exponential type was considered. Thus,  we clarify   results  \cite{firstab_lit:1Lidskii} devoted to the decomposition on the root vector system of the non-selfadjoint operator. We use  a technique of the entire function theory and introduce  a so-called  Schatten-von Neumann class of the convergence  exponent. Considering strictly accretive operators satisfying special conditions formulated in terms of the norm,   using a  sequence of contours of the power type,  we invent a peculiar  method how to calculate  a contour integral involved in the problem in its general statement.     Finally, we produce applications to differential equations in the abstract Hilbert space.
  In particular,     the existence and uniqueness theorems  for evolution equations   with the right-hand side -- a  differential operator  with a fractional derivative in  final terms are covered by the invented abstract method. In this regard such operators as a Riemann-Liouville  fractional differential operator,    Kipriyanov operator, Riesz potential,  difference operator are involved.
Note that analysis of the required  conditions imposed upon the right-hand side of the    evolution equations that are in the scope  leads us to   relevance of the central idea of the paper. In this regard we should note  a well-known fact  \cite{firstab_lit:Shkalikov A.},\cite{firstab_lit(arXiv non-self)kukushkin2018}   that a particular interest appears in the case when a senior term of the operator at least
    is not selfadjoint  for in the contrary case there is a plenty of results devoted to the topic wherein  the following papers are well-known
\cite{firstab_lit:1Katsnelson},\cite{firstab_lit:1Krein},\cite{firstab_lit:Markus Matsaev},\cite{firstab_lit:2Markus},\cite{firstab_lit:Shkalikov A.}. The fact is that most of them deal with a decomposition of the  operator  to a sum  where the senior term
     must be either a selfadjoint or normal operator. In other cases the  methods of the papers
     \cite{kukushkin2019}, \cite{firstab_lit(arXiv non-self)kukushkin2018} become relevant   and allow us  to study spectral properties of   operators  whether we have the mentioned above  representation or not. Here, we should remark that the results of  the papers \cite{firstab_lit:2Agranovich1994},\cite{firstab_lit:Markus Matsaev} applicable to study non-selfadjoin operators  are   based on the sufficiently strong assumption regarding the numerical range of values of the operator, whereas
the methods  \cite{firstab_lit(arXiv non-self)kukushkin2018}   can be  used    in  the  natural way,  if we deal with more abstract constructions formulated in terms of the semigroup theory   \cite{kukushkin2021a}.  The  central challenge  of the latter  paper  is how  to create a model   representing     a  composition of  fractional differential operators   in terms of the semigroup theory. Here we should note that motivation arouse in connection with the fact that
  a second order differential operator can be presented  as a some kind of  a  transform of   the infinitesimal generator of a shift semigroup and stress that
  the eigenvalue problem for the operator
     was previously  studied by methods of  theory of functions   \cite{firstab_lit:1Nakhushev1977}, \cite{firstab_lit:1Aleroev1994}.
Having been inspired by   novelty of the  idea  we generalize a   differential operator with a fractional integro-differential composition  in the final terms   to some transform of the corresponding  infinitesimal generator of the shift semigroup.
By virtue of the   methods obtained in the paper
\cite{firstab_lit(arXiv non-self)kukushkin2018} we   managed  to  study spectral properties of the  infinitesimal generator  transform and obtained an outstanding result --
   asymptotic equivalence between   the
real component of the resolvent and the resolvent of the   real component of the operator. The relevance is based on the fact that
   the  asymptotic formula  for   the operator  real component  can be  established in most  cases due to well-known asymptotic relations  for the regular differential operators as well as for the singular ones
 \cite{firstab_lit:Rosenblum}. It is remarkable that the results establishing spectral properties of non-selfadjoint operators  allow  us to implement a novel approach to the problem  of the basis property of the root vectors of non-selfadjoint operators.
   In its own turn the application  of results connected with the basis property  covers  many   problems in the framework of the theory of evolution  equations. However,   as a main advantage, we establish an abstract formula for the solution.   Moreover,  the norm-convergence of the series representing the solution allows us to apply the methods of the approximation theory. Thus, we can claim that  the offered approach is undoubtedly novel and relevant. As an application we produce the artificially constructed normal operator. This example  indicates  the relevance and significance of a variant of the natural  lacunae method    allowing  us to formulate the optimal conditions in comparison with the Lidskii    results \cite{firstab_lit:1Lidskii}.

\section{Preliminaries}

Let    $ C,C_{i} ,\;i\in \mathbb{N}_{0}$ be   real constants. We   assume   that  a  value of $C$ is positive and   can be different in   various formulas  but   values of $C_{i} $ are  certain. Denote by $ \mathrm{int} \,M,\;\mathrm{Fr}\,M$ the interior and the set of boundary points of the set $M$ respectively.   Everywhere further, if the contrary is not stated, we consider   linear    densely defined operators acting on a separable complex  Hilbert space $\mathfrak{H}$. Denote by $ \mathcal{B} (\mathfrak{H})$    the set of linear bounded operators   on    $\mathfrak{H}.$  Denote by
    $\tilde{L}$   the  closure of an  operator $L.$ We establish the following agreement on using  symbols $\tilde{L}^{i}:= (\tilde{L})^{i},$ where $i$ is an arbitrary symbol.  Denote by    $    \mathrm{D}   (L),\,   \mathrm{R}   (L),\,\mathrm{N}(L)$      the  {\it domain of definition}, the {\it range},  and the {\it kernel} or {\it null space}  of an  operator $L$ respectively. The deficiency (codimension) of $\mathrm{R}(L),$ dimension of $\mathrm{N}(L)$ are denoted by $\mathrm{def}\, T,\;\mathrm{nul}\,T$ respectively. Assume that $L$ is a closed   operator acting on $\mathfrak{H},\,\mathrm{N}(L)=0,$  let us define a Hilbert space
$
 \mathfrak{H}_{L}:= \big \{f,g\in \mathrm{D}(L),\,(f,g)_{ \mathfrak{H}_{L}}=(Lf,Lg)_{\mathfrak{H} } \big\}.
$
Consider a pair of complex Hilbert spaces $\mathfrak{H},\mathfrak{H}_{+},$ the notation
$
\mathfrak{H}_{+}\subset\subset\mathfrak{ H}
$
   means that $\mathfrak{H}_{+}$ is dense in $\mathfrak{H}$ as a set of    elements and we have a bounded embedding provided by the inequality
$$
\|f\|_{\mathfrak{H}}\leq C_{0}\|f\|_{\mathfrak{H}_{+}},\,C_{0}>0,\;f\in \mathfrak{H}_{+},
$$
moreover   any  bounded  set with respect to the norm $\mathfrak{H}_{+}$ is compact with respect to the norm $\mathfrak{H}.$
  Let $L$ be a closed operator, for any closable operator $S$ such that
$\tilde{S} = L,$ its domain $\mathrm{D} (S)$ will be called a core of $L.$ Denote by $\mathrm{D}_{0}(L)$ a core of a closeable operator $L.$ Let    $\mathrm{P}(L)$ be  the resolvent set of an operator $L$ and
     $ R_{L}(\zeta),\,\zeta\in \mathrm{P}(L),\,[R_{L} :=R_{L}(0)]$ denotes      the resolvent of an  operator $L.$ Denote by $\lambda_{i}(L),\,i\in \mathbb{N} $ the eigenvalues of an operator $L.$
 Suppose $L$ is  a compact operator and  $N:=(L^{\ast}L)^{1/2},\,r(N):={\rm dim}\,  \mathrm{R}  (N);$ then   the eigenvalues of the operator $N$ are called   the {\it singular  numbers} ({\it s-numbers}) of the operator $L$ and are denoted by $s_{i}(L),\,i=1,\,2,...\,,r(N).$ If $r(N)<\infty,$ then we put by definition     $s_{i}=0,\,i=r(N)+1,2,...\,.$
 According  to the terminology of the monograph   \cite{firstab_lit:1Gohberg1965}  the  dimension  of the  root vectors subspace  corresponding  to a certain  eigenvalue $\lambda_{k}$  is called  the {\it algebraic multiplicity} of the eigenvalue $\lambda_{k}.$
Let  $\nu(L)$ denotes   the sum of all algebraic multiplicities of an  operator $L.$ Denote by $n(r)$ a function equals to the quantity of the elements of the sequence $\{a_{n}\}_{1}^{\infty},\,|a_{n}|\uparrow\infty$ within the circle $|z|<r.$ Let $A$ be a compact operator, denote by $n_{A}(r)$   {\it counting function}   a function $n(r)$ corresponding to the sequence  $\{s^{-1}_{i}(A)\}_{1}^{\infty}.$ Let  $\mathfrak{S}_{p}(\mathfrak{H}),\, 0< p<\infty $ be       a Schatten-von Neumann    class and      $\mathfrak{S}_{\infty}(\mathfrak{H})$ be the set of compact operators.
    Suppose  $L$ is  an   operator with a compact resolvent and
$s_{n}(R_{L})\leq   C \,n^{-\mu},\,n\in \mathbb{N},\,0\leq\mu< \infty;$ then
 we
 denote by  $\mu(L) $   order of the     operator $L$ in accordance with  the definition given in the paper  \cite{firstab_lit:Shkalikov A.}.
 Denote by  $ \mathfrak{Re} L  := \left(L+L^{*}\right)/2,\, \mathfrak{Im} L  := \left(L-L^{*}\right)/2 i$
  the  real  and   imaginary components    of an  operator $L$  respectively.
In accordance with  the terminology of the monograph  \cite{firstab_lit:kato1980} the set $\Theta(L):=\{z\in \mathbb{C}: z=(Lf,f)_{\mathfrak{H}},\,f\in  \mathrm{D} (L),\,\|f\|_{\mathfrak{H}}=1\}$ is called the  {\it numerical range}  of an   operator $L.$
  An  operator $L$ is called    {\it sectorial}    if its  numerical range   belongs to a  closed
sector     $\mathfrak{ L}_{\iota}(\theta):=\{\zeta:\,|\arg(\zeta-\iota)|\leq\theta<\pi/2\} ,$ where      $\iota$ is the vertex   and  $ \theta$ is the semi-angle of the sector   $\mathfrak{ L}_{\iota}(\theta).$ If we want to stress the  correspondence  between $\iota$ and $\theta,$  then   we will write $\theta_{\iota}.$
 An operator $L$ is called  {\it bounded from below}   if the following relation  holds  $\mathrm{Re}(Lf,f)_{\mathfrak{H}}\geq \gamma_{L}\|f\|^{2}_{\mathfrak{H}},\,f\in  \mathrm{D} (L),\,\gamma_{L}\in \mathbb{R},$  where $\gamma_{L}$ is called a lower bound of $L.$ An operator $L$ is called  {\it   accretive}   if  $\gamma_{L}=0.$
 An operator $L$ is called  {\it strictly  accretive}   if  $\gamma_{L}>0.$      An  operator $L$ is called    {\it m-accretive}     if the next relation  holds $(A+\zeta)^{-1}\in \mathcal{B}(\mathfrak{H}),\,\|(A+\zeta)^{-1}\| \leq   (\mathrm{Re}\zeta)^{-1},\,\mathrm{Re}\zeta>0. $
An operator $L$ is called    {\it m-sectorial}   if $L$ is   sectorial    and $L+ \beta$ is m-accretive   for some constant $\beta.$   An operator $L$ is called     {\it symmetric}     if one is densely defined and the following  equality  holds $(Lf,g)_{\mathfrak{H}}=(f,Lg)_{\mathfrak{H}},\,f,g\in   \mathrm{D}  (L).$

    Consider a   sesquilinear form   $ t  [\cdot,\cdot]$ (see \cite{firstab_lit:kato1980} )
defined on a linear manifold  of the Hilbert space $\mathfrak{H}.$   Denote by $   t  [\cdot ]$ the  quadratic form corresponding to the sesquilinear form $ t  [\cdot,\cdot].$
Let   $  \mathfrak{h}=( t + t ^{\ast})/2,\, \mathfrak{k}   =( t - t ^{\ast})/2i$
   be a   real  and    imaginary component     of the   form $  t $ respectively, where $ t^{\ast}[u,v]=t \overline{[v,u]},\;\mathrm{D}(t ^{\ast})=\mathrm{D}(t).$ According to these definitions, we have $
 \mathfrak{h}[\cdot]=\mathrm{Re}\,t[\cdot],\,  \mathfrak{k}[\cdot]=\mathrm{Im}\,t[\cdot].$ Denote by $\tilde{t}$ the  closure   of a   form $t.$  The range of a quadratic form
  $ t [f],\,f\in \mathrm{D}(t),\,\|f\|_{\mathfrak{H}}=1$ is called    {\it range} of the sesquilinear form  $t $ and is denoted by $\Theta(t).$
 A  form $t$ is called    {\it sectorial}    if  its    range  belongs to   a sector  having  a vertex $\iota$  situated at the real axis and a semi-angle $0\leq\theta_{\iota}<\pi/2.$   Suppose   $t$ is a closed sectorial form; then  a linear  manifold  $\mathrm{D}_{0}(t) \subset\mathrm{D} (t)$   is
called    {\it core}  of $t,$ if the restriction   of $t$ to   $\mathrm{D}_{0}(t)$ has the   closure
$t$ (see\cite[p.166]{firstab_lit:kato1980}).   Due to Theorem 2.7 \cite[p.323]{firstab_lit:kato1980}  there exist unique    m-sectorial operators  $T_{t},T_{ \mathfrak{h}} $  associated  with   the  closed sectorial   forms $t,  \mathfrak{h}$   respectively.   The operator  $T_{  \mathfrak{h}} $ is called  a {\it real part} of the operator $T_{t}$ and is denoted by  $Re\, T_{t}.$ Suppose  $L$ is a sectorial densely defined operator and $t[u,v]:=(Lu,v)_{\mathfrak{H}},\,\mathrm{D}(t)=\mathrm{D}(L);$  then
 due to   Theorem 1.27 \cite[p.318]{firstab_lit:kato1980}   the corresponding  form $t$ is   closable, due to
   Theorem 2.7 \cite[p.323]{firstab_lit:kato1980} there exists   a unique m-sectorial operator   $T_{\tilde{t}}$   associated  with  the form $\tilde{t}.$  In accordance with the  definition \cite[p.325]{firstab_lit:kato1980} the    operator $T_{\tilde{t}}$ is called     a {\it Friedrichs extension} of the operator $L.$
Everywhere further,   unless  otherwise  stated,  we   use  notations of the papers   \cite{firstab_lit:1Gohberg1965},  \cite{firstab_lit:kato1980},  \cite{firstab_lit:kipriyanov1960}, \cite{firstab_lit:1kipriyanov1960},
\cite{firstab_lit:samko1987}.

\vspace{0.5cm}
\noindent{\bf   Some properties of non-selfadjoint operators }\\

In this section we  explore  a special  operator class for which    a number of  spectral theory theorems can be applied.   As an application of the obtained abstract results  we study a basis property of the root vectors of the operator in terms of the order of the  operator  real part.   By virtue of such an approach we express a convergence exponent of $ {\it s} $-numbers  through the order of the operator  real part.
The theorem given bellow gives us a description of spectral properties of some class of non-selfadjoint operators.
\begin{teo}\label{T1}
Assume that $L$ is a non-sefadjoint operator acting in $\mathfrak{H},$  the following  conditions hold\\

 \noindent  $ (\mathrm{H}1) $ There  exists a Hilbert space $\mathfrak{H}_{+}\subset\subset\mathfrak{ H}$ and a linear manifold $\mathfrak{M}$ that is  dense in  $\mathfrak{H}_{+}.$ The operator $L$ is defined on $\mathfrak{M}.$    \\

 \noindent  $( \mathrm{H2} )  \,\left|(Lf,g)_{\mathfrak{H}}\right|\! \leq \! C_{1}\|f\|_{\mathfrak{H}_{+}}\|g\|_{\mathfrak{H}_{+}},\,
      \, \mathrm{Re}(Lf,f)_{\mathfrak{H}}\!\geq\! C_{2}\|f\|^{2}_{\mathfrak{H}_{+}} ,\,f,g\in  \mathfrak{M},\; C_{1},C_{2}>0.
$
\\

\noindent Let $W$ be a restriction of the operator  $L$ on the set $\mathfrak{M}.$  Then   the following  propositions are true.\\

\noindent$({\bf A})$ {\it We have the following classification
\begin{equation*}
R_{ \tilde{W} }\in  \mathfrak{S}_{p},\,p= \left\{ \begin{aligned}
\!l,\,l>2/\mu,\,\mu\leq1,\\
   1,\,\mu>1    \\
\end{aligned}
 \right.\;,
\end{equation*}
here and further $\mu$ is  the   order of $H:=Re \, \tilde{W} .$
Moreover  under  the  assumptions $ \lambda_{n}(R_{H})\geq  C \,n^{-\mu},\,n\in \mathbb{N},$  we have the following implication
$$
 \left\{R_{  \tilde{W}}\in\mathfrak{S}_{p},\;1\leq p<\infty\right\}\;  \Rightarrow \;\mu p>1.
$$
}\\

\noindent$({\bf B})$ {\it The following relation  holds
 \begin{equation*}
\sum\limits_{i=1}^{n}|\lambda_{i}(R_{ \tilde{W}})|^{p}\leq
  C \sum\limits_{i=1}^{n } \,\lambda^{ p}_{i}(R_{H})  ,\,1\leq p<\infty,\, \;(n=1,2,...,\, \nu(R_{ \tilde{W} })) ,
\end{equation*}
moreover   if  $\nu(R_{ \tilde{W} })=\infty$ and      $\mu \neq0,$  then  the following asymptotic formula  holds
$$
|\lambda_{i}(R_{\tilde{W}})|=  o\left(i^{-\tau}    \right)\!,\,i\rightarrow \infty,\;0<\tau<\mu.
$$}
\noindent $({\bf C})$  {\it Assume that   $\theta< \pi \mu/2\, ,$    where $\theta$ is   the   semi-angle of the     sector $ \mathfrak{L}_{0}(\theta)\supset \Theta (\tilde{W}).$
Then  the system of   root   vectors  of   $R_{ \tilde{W} }$ is complete in $\mathfrak{H}.$}
\end{teo}
  Throughout the paper we  formulate results in terms of the  restriction   $W$  on the set $\mathfrak{M}$ of the operator  $L$ satisfying  the Theorem \ref{T1} conditions.      We also    use  the  short-hand  notations $A:=R_{\tilde{W}},\,\mu:=\mu(H),$ where $H:=Re \tilde{W}.$\\

\vspace{0.5cm}

\noindent{\bf Some facts of the entire functions theory}\\

Here we introduce some notions and facts of the entire functions theory, we follow the monograph \cite{firstab_lit:Eb. Levin}. In this subsection, we   use the following notations
$$
G(z,p):=(1-z)e^{z+\frac{z^{2}}{2}+...+\frac{z^{p}}{p}},\,G(z,0):=(1-z).
$$
Consider such an entire function  that its zeros satisfy the following relation for some   $\lambda>0$
\begin{equation}\label{1}
 \sum\limits_{n=1}^{\infty}\frac{1}{|a_{n}|^{\lambda}}<\infty.
\end{equation}
In this case we denote by $p$ the smallest integer number for which the following condition holds
\begin{equation}\label{2}
\sum\limits_{n=1}^{\infty}\frac{1}{|a_{n}|^{p+1}}<\infty .
\end{equation}
 It is clear that $0\leq p<\lambda.$ It is proved that under the assumption \eqref{1}  the   infinite product
 \begin{equation}\label{3}
 \prod\limits_{n=1}^{\infty}(z): =\prod\limits_{n=1}^{\infty} G\left(\frac{z}{a_{n}},p\right)
\end{equation}
 is uniformly convergent, we will call it a canonical product and call $p$ the genus of the canonical product.
By the   {\it convergence exponent} $\rho$ of the sequence
$
\{a_{n}\}_{1}^{\infty}\subset \mathbb{C},\,a_{n}\neq 0,\,a_{n}\rightarrow \infty
$
 we mean the greatest lower bound for numbers $\lambda$ for which   series \eqref{1} converges.
 Note that if $\lambda$ equals to a convergence  exponent then series \eqref{1} may or  may not be convergent. For instance, the sequences $a_{n}= 1/n^{\lambda}$ and $1/(n\ln^{2} n)^{\lambda}$ have the same convergence exponent $\lambda=1,$ but in the first case the series \eqref{1} is divergent when $\lambda=1$ while in the second one it is convergent. In this paper, we have a special interest regarding the first case. Consider the following obvious relation between the convergence exponent $\rho$ and the genus $p$ of the corresponding canonical product $p \leq\rho\leq p+1.$ It is clear that if $\rho=p,$   then    the series  \eqref{1} diverges for $\lambda=\rho,$ while $\rho=p+1$ means that the series converges  (in accordance with the definition of $p$). In the monograph \cite{firstab_lit:Eb. Levin}, it is considered  a more precise characteristic of the density of the sequence $\{a_{n}\}_{1}^{\infty}$ than the convergence exponent. Thus, there is defined  a so-called counting function $n(r)$ equals to a number of points of the sequence in the circle $|z|<r.$ By upper density of the sequence, we call a number
 $$
 \Delta=\overline{\lim\limits_{r\rightarrow\infty}} n(t)/t^{\rho}.
 $$
 If a limit exists in the ordinary sense (not in the  sense of the upper limit),  then $\Delta$ is called the density. Note that it is proved in Lemma 1
  \cite{firstab_lit:Eb. Levin} that
 $$
   \lim\limits_{r\rightarrow\infty}  n(t)/t^{\rho+\varepsilon}\rightarrow 0,\,\varepsilon>0.
 $$
 We need the following fact (see \cite{firstab_lit:Eb. Levin} Lemma 3).
 \begin{lem}\label{L1}
 If the series \eqref{2} converges, then the corresponding infinite product \eqref{3}  satisfies  the following inequality in the entire complex plane
 $$
 \ln\left|  \prod\limits_{n=1}^{\infty}(z)\right|\leq Cr^{p}\left(\int\limits_{0}^{r}\frac{n(t)}{t^{p+1}}dt+r\int\limits_{r}^{\infty}\frac{n(t)}{t^{p+2}}dt\right),\,r:=|z|.
 $$
\end{lem}
Using this result, it is not hard to prove a relevant  fact mentioned in the monograph \cite{firstab_lit:Eb. Levin}. Since it has a principal role in the further narrative, then  we formulate it as a lemma  in terms of the   density.
\begin{lem}\label{L2}
Assume that the following series is convergent for some values $\lambda>0$ i.e.
$$
\sum\limits_{n=1}^{\infty}\frac{1}{|a_{n}|^{\lambda}}<\infty.
$$
Then  the following  relation holds
\begin{equation}\label{4}
\left|\prod\limits_{n=1}^{\infty} (z)\right|\leq e^{\beta(r)r^{\rho_{1}}},\,\beta(r)= r^{p-\rho_{1} }\left(\int\limits_{0}^{r}\frac{n(t)}{t^{p+1 }}dt+
r\int\limits_{r}^{\infty}\frac{n(t)}{t^{p+2  }}dt\right).
\end{equation}
  In the case $\rho_{1}=\rho,$ we have   $\beta(r)\rightarrow 0,$ if at list one of the following conditions holds:  the convergence exponent $\rho<\lambda$  is   non-integer       and the density equals zero,  the convergence exponent $\rho=\lambda$  is   arbitrary.    In addition,   the  equality  $\rho=\lambda$ guaranties that the density equals zero. In the case $\rho_{1}>\rho,$ we claim that $\beta(r)\rightarrow 0,$ without any additional conditions.
\end{lem}
\begin{proof}
Applying  Lemma \ref{L1}, we establish relation \eqref{4}. Consider a case   when  $\rho<\lambda$ is non-integer.
Taking into account the fact that the density equals  zero, using   L'H\^{o}pital's rule,    we easily obtain
\begin{equation}\label{5}
 r^{p-\rho}  \int\limits_{0}^{r}\frac{n(t)}{t^{p+1}}dt \rightarrow 0 ;\;r^{p+1-\rho } \int\limits_{r}^{\infty}\frac{n(t) }{t^{p+2}}dt \rightarrow 0,
\end{equation}
(here we should remark  that if $\rho$ is integer, then $p=\rho$).
Therefore  $\beta(r)\rightarrow 0.$
Consider the case when  $\rho=\lambda,$ then let us rewrite the series \eqref{1} in the form of the Stiltes integral
$$
\sum\limits_{n=1}^{\infty}\frac{1}{|a_{n}|^{\lambda}}=\int\limits_{0}^{\infty}\frac{d n(t)}{t^{\rho}} .
$$
Using integration by parts formulae, we get
$$
\int\limits_{0}^{r}\frac{d n(t)}{t^{\rho}} =\frac{n(r)}{r^{\rho}} -\frac{n(\gamma)}{C^{\rho}}+\rho\int\limits_{0}^{r}\frac{n(t)}{t^{\rho+1}}dt.
$$
Here,  we should note that there exists a neighborhood of the point zero in which $n(t)=0.$
The latter representation shows us that the following  integral converges i.e.
$$
\int\limits_{0}^{\infty}\frac{n(t)}{t^{\rho+1}}dt<\infty.
$$
In its own turn, it follows that
$$
\frac{n(r)}{r^{\rho}}=n(r)\rho\int\limits_{r}^{\infty}\frac{1}{t^{\rho+1}}dt<\rho\int\limits_{r}^{\infty}\frac{n(t)}{t^{\rho+1}}dt\rightarrow 0,\,r\rightarrow \infty.
$$
 Using this fact, analogously to the above applying L'H\^{o}pital's rule, we conclude  that  \eqref{5} holds if $\rho=\lambda$ is non-integer. If   $\rho=\lambda$ is  integer  then it is clear that  we have $\rho=p+1,$ here we should remind  that it is not possible to assume $\rho=p$ due to the definition of $p.$ In the case $\rho=p+1,$    using the above  reasonings, we get
$$
 r^{-1}  \int\limits_{0}^{r}\frac{n(t)}{t^{p+1}}dt \rightarrow 0 ;\; \int\limits_{r}^{\infty}\frac{n(t) }{t^{p+2}}dt \rightarrow 0,
$$
from what follows  the fact that $\beta(r)\rightarrow0$. The reasonings related to the  case $\rho_{1}>\rho$ is absolutely analogous,     we  left  the proof to  the reader. The proof is complete.
\end{proof}
\begin{lem} \label{L3} We claim that the following implication holds
$$
\ln r\frac{n(r)}{r^{\rho_{1}}}\rightarrow0,\,\Longrightarrow \beta(r)\ln r \rightarrow 0,\,
$$
where
$$
\beta(r)= r^{p-\rho_{1} }\left(\int\limits_{0}^{r}\frac{n(t)}{t^{p+1 }}dt+
r\int\limits_{r}^{\infty}\frac{n(t)}{t^{p+2  }}dt\right),\,\rho_{1}\neq p,\,\rho_{1}\neq p+1.
$$
\end{lem}
\begin{proof}
To prove the fact   $\beta(r)\ln r\rightarrow 0,$  we should consider   representation \eqref{4}, we have
$$
\beta(r)\ln r=  r^{p-\rho }\left(\int\limits_{0}^{r}\frac{n(t)}{t^{p+1 }}dt+
r\int\limits_{r}^{\infty}\frac{n(t)}{t^{p+2  }}dt\right) \ln r.
$$
Let us define the following auxiliary functions
$$
u_{1}(r):=\ln r\int\limits_{0}^{r}\frac{n(t)}{t^{p+1 }}dt,\,u_{2}(r):=r^{\rho-p} ,\,v_{1}(r):=\ln r \int\limits_{r}^{\infty}\frac{n(t)}{t^{p+2  }}dt,\,v_{2}(r):=r^{\rho-p-1}.
$$
It is clear that
$$
u_{1}'(r):=\frac{1}{r} \int\limits_{0}^{r}\frac{n(t)}{t^{p+1 }}dt +\ln r \frac{n(r)}{r^{p+1 }};\;
v'_{1}(r):=\frac{1}{r} \int\limits_{r}^{\infty}\frac{n(t)}{t^{p+2  }}dt+ \ln r\frac{n(r)}{r^{p+2  }}.
$$
Therefore
\begin{equation}\label{4d}
\frac{u_{1}'(r)}{u_{2}'(r)}=Cr^{p-\rho} \int\limits_{0}^{r}\frac{n(t)}{t^{p+1 }}dt +C\ln r \frac{n(r)}{r^{\rho}};\;\frac{v_{1}'(r)}{v_{2}'(r)}=Cr^{p+1-\rho} \int\limits_{r}^{\infty}\frac{n(t)}{t^{p+2 }}dt +C\ln r \frac{n(r)}{r^{\rho}}.
\end{equation}
Notice that
$\beta(r)\ln r=  u_{1} (r)/u_{2} (r)+v_{1} (r)/v_{2} (r),$
  applying    L'H\^{o}pital's rule, we have
\begin{equation}\label{5d}
  \beta(r)\ln r\sim  \frac{u_{1}'(r)}{u_{2}'(r)}+ \frac{v_{1}'(r)}{v_{2}'(r)},\,r\rightarrow\infty.
\end{equation}
In an analogous way, we obtain the following implication
\begin{equation}\label{6d}
\frac{n(r)}{r^{\rho}} \rightarrow 0,\;\Longrightarrow\left\{r^{p-\rho} \int\limits_{0}^{r}\frac{n(t)}{t^{p+1 }}dt\rightarrow 0;\;r^{p+1-\rho} \int\limits_{r}^{\infty}\frac{n(t)}{t^{p+2 }}dt\rightarrow0 \right\}.
\end{equation}
Thus, taking into account the premise
$
\ln r \cdot  n(r)/r^{\rho} \rightarrow0,
$ combining \eqref{4d}, \eqref{5d}, \eqref{6d},
we obtain the desired result.
\end{proof}
Regarding Lemma \ref{L3}, we can produce the following example that indicates the relevance of the issue itself.

\begin{ex}\label{E1}
  There exists a sequence $\{a_{n}\}_{1}^{\infty}$ such that  density equals zero, moreover
$$
 \beta(r)\ln r\rightarrow 0,\;\sum\limits_{n=1}^{\infty}\frac{1}{|a_{n}|^{\rho}}=\infty.
$$
We can construct the required sequence supposing $n(r)\sim r^{\rho}(\ln r \cdot \ln\ln r )^{-1},\,\rho>0.
$
It  follows from the latter relation  directly   that the density equals zero.
It is clear that we can represent partial sums of series \eqref{1} due to the Stiltes integral
$$
 \sum\limits_{n=1}^{k}\frac{1}{|a_{n}|^{\lambda}}= \int\limits_{0}^{r(k) }\frac{d n(t)}{t^{\lambda}},\,\lambda\geq\rho.
$$
Thus  the sequence $\{a_{n}\}_{1}^{\infty}$ is defined by the function $n(r).$
Applying the integration by parts formulae, we get
$$
 \int\limits_{0}^{r }\frac{d n(t)}{t^{\lambda}}  = \frac{n(r)}{r^{\lambda}} - \frac{1}{a^{\lambda}_{1}} +\lambda\int\limits_{0}^{r}\frac{n(t)}{t^{\lambda+1}}dt .
$$
Using the latter relation,  we can easily establish the fact that the density equals zero while the last integral is divergent when $\lambda=\rho ,\;r\rightarrow\infty,$ we have
$$
 \int\limits_{0}^{r}\frac{n(t)}{t^{\rho +1}}dt\geq C \int\limits_{0}^{r}\frac{ dt}{t\ln t\cdot \ln\ln t }= \ln\ln\ln r-C.
$$
On the other hand, we have
$$
 \int\limits_{0}^{\infty}\frac{n(t)}{t^{\lambda+1}}dt=  \int\limits_{0}^{\infty}\frac{ dt}{t^{1+\lambda-\rho}\ln t\cdot \ln\ln t }< \infty ,\; \lambda>\rho.
$$
Therefore,  the series \eqref{1} is divergent if $\lambda=\rho$ and   convergent if $\lambda=\rho+\varepsilon,\,\varepsilon>0.$ Thus,   the denotation  $\rho$ is justified.   We should explain that in these reasonings    $\rho$ has two meanings: a power and a  convergence exponent,  we have established the identity of them.
Let us  prove the fact   $\beta(r)\ln r\rightarrow 0,$  for this purpose in accordance with Lemma \ref{L3},
  it suffices to show that
$$
\ln r \frac{n(r)}{r^{\rho}}\rightarrow0,\;r\rightarrow\infty.
$$
Substituting $r^{\rho}(\ln r \cdot \ln\ln r )^{-1}$  instead of $n(r),$ we get
$$
\ln r \frac{n(r)}{r^{\rho}}=   \frac{1 }{  \ln\ln r }\rightarrow0,\;r\rightarrow\infty.
$$
Thus, we have established  the fulfilment of the made claims.
\end{ex}

\noindent{\bf Schatten-von Neumann  class and   the particular case corresponding to the normal operator}\\

 Let  $\mathfrak{S}_{q}(\mathfrak{H}),\, 0< q<\infty $ be       a Schatten-von Neumann    class and      $\mathfrak{S}_{\infty}(\mathfrak{H})$ be the set of compact operators. By definition, put
$$
\mathfrak{S}_{q}(\mathfrak{H}):=\left\{ T: \mathfrak{H}\rightarrow \mathfrak{H},  \sum\limits_{i=1}^{\infty}s^{q}_{i}(L)<\infty,\;0< q<\infty \right\}.
$$
  Denote by $\tilde{\mathfrak{S}}_{\rho}(\mathfrak{H})$ the class of the operators such that
$$
 \tilde{\mathfrak{S}}_{\rho}(\mathfrak{H}):=\{T\in\mathfrak{S}_{\rho+\varepsilon},\,T\, \overline{\in} \,\mathfrak{S}_{\rho-\varepsilon},\,\forall\varepsilon>0 \}.
$$
This operator class we will call      {\it Schatten-von Neumann    class  of the convergence exponent.}
Note  that there exists a one to one correspondence between selfadjoint compact operators and monotonically decreasing sequences.  If we consider example \ref{E1},   then we see that the made definition becomes relevant  in this regard.

\begin{lem}\label{L4}
Assume that
 $$(\ln^{\kappa+1}x)'_{\lambda_{i}(H)}  =o(i^{-\kappa})    ,\;\kappa\in(0,1].$$
 Then in the general case, we get
$$
A\in  \tilde{\mathfrak{S}}_{\rho} ,\;\rho\in[0,2/\kappa] ,\; n_{A}(r)=o( r^{ 2/\kappa}/ \ln r).
$$
In the particular case, when $\tilde{W}$ is normal,  we get
$$
A\in \tilde{\mathfrak{S}}_{\rho} ,\;\rho\in[0,1/\kappa],\; n_{A}(r)=o( r^{1/\kappa}/\ln r).
$$
Moreover, the additional assumption
 $  \lambda_{i}(H) =O(i^{\kappa+\varepsilon}),\,\forall\varepsilon>0$ gives us  the estimate    $\rho\geq 1/\kappa$  in both cases, thus in the case when $\tilde{W}$ is normal, we get
 $$
 A\in  \tilde{\mathfrak{S}}_{1/\kappa}.
 $$
\end{lem}
\begin{proof} Note that the fact $
A\in  \tilde{\mathfrak{S}}_{\rho} ,\, \rho\in[0, 2/\kappa],
$
follows directly from the Theorem \ref{T1}, claim $({\bf A}).$
In accordance with  relation (54) \cite{firstab_lit(arXiv non-self)kukushkin2018}, we have
$
(|A |^{2} f,f)_{\mathfrak{H}}=\|Af\|^{2}_{\mathfrak{H}}\leq C\cdot{\rm Re}(Af,f)_{\mathfrak{H}}= C(V f,f)_{\mathfrak{H}},
$
where $V:=  (A+A^{\ast})/2.$ In accordance with the  Theorem 5 \cite{firstab_lit(arXiv non-self)kukushkin2018}, we have
$
\lambda_{i}(V)\asymp \lambda_{i}(R_{H}),
$
thus we have $s_{i}(A)\leq C\lambda^{1/2}_{i}(R_{H});\,s^{-1}_{i}(A)\geq C\lambda^{1/2}_{i}( H ),$ the detailed proof of the latter fact see in the Theorem 7 \cite{firstab_lit(arXiv non-self)kukushkin2018}. Using the monotonous property of the functions, we have
$$
 \frac{\ln^{\kappa } s^{-1}_{i}(A)}{s^{-2}_{i}(A)} \leq C\cdot  \frac{\ln^{\kappa} \lambda_{i}(H)}{ \lambda_{i}(H) } \leq C\cdot \frac{\alpha_{i}}{i^{\kappa }},
$$
where $\alpha_{i}\rightarrow 0.$
Hence
$$
\frac{i\ln s^{-1}_{i}(A)}{s^{-2/\kappa}_{i}(A)}\leq C\cdot  \alpha^{1/\kappa}_{i}.
$$
Taking into account the facts
$
n(s^{-1}_{i})=i;\,n(r)= n(s^{-1}_{i}),\,s^{-1}_{i}<r<s^{-1}_{i+1},
$
using the monotonous property of the functions, we get
$$
\frac{n(r)\ln r}{r^{2/\kappa}}  <  C\cdot  \alpha_{i},\;s^{-1}_{i}<r<s^{-1}_{i+1}.
$$
The proof corresponding to the general case  is complete. Assume  that the operator $\tilde{W}$ is normal, then it is not hard to prove that $A$ is normal also. Let us show that the   operator $V:=(A+A^{\ast})/2$ has a complete orthonormal  system of the  eigenvectors.   Using   formula  (53) \cite{firstab_lit(arXiv non-self)kukushkin2018},   we get
$$
V^{\!^{-1}}=2H^ \frac{1 }{2}  (I+B^{2})      H^{ \frac{1}{2}}.
$$
Note that in accordance with relation (67)  \cite{firstab_lit(arXiv non-self)kukushkin2018}, we have
\begin{equation}\label{6a}
 (V^{^{\!-\!1}}\!\!f,f)_{\mathfrak{H}}  =2(  S      H^{ \frac{1}{2}}f,H^ \frac{1 }{2}f)_{\mathfrak{H}}\geq 2\|H^ \frac{1 }{2} f\|^{2}_{\mathfrak{H}}=2(Hf,f)_{\mathfrak{H}}  ,\,f\in \mathrm{D}(V^{^{\!-\!1}}),
\end{equation}
where $S=I+B^{2}.$
 Since  $V$ is selfadjoint, then   due to  Theorem 3 \cite[p.136]{firstab_lit: Ahiezer1966} the operator    $V^{^{\!-\!1}}$ is selfadjoint also.
  Combining  \eqref{6a} with
 Lemma   3   \cite{firstab_lit(arXiv non-self)kukushkin2018},    we get that    $V^{^{\!-\!1}}$ is strictly accretive.
Using these facts   we can write
\begin{equation*}
\|f\| _{V^{^{ -\!1}}} \geq C\|f\|_{H  } ,\,f\in  \mathfrak{H}_{ V^{^{ -\!1}}},
\end{equation*}
where the above norms are understood as the norms of the energetic spaces generated by the operators $V^{^{ -\!1}}$ and $H$ respectively.
Since  the operator $H$ has a discrete spectrum (see Theorem 5.3 \cite{firstab_lit:1kukushkin2018}), then any set  bounded   with respect to the      norm $\mathfrak{H}_{H}$ is a compact set   with respect to the norm    $\mathfrak{H}$ (see Theorem 4 \cite[p.220]{firstab_lit:mihlin1970}). Combining this fact with \eqref{6a},
 Theorem 3 \cite[p. 216]{firstab_lit:mihlin1970}, we get
  that the operator $V^{^{\!-\!1}}$ has a discrete spectrum, i.e.   it has   the infinite set of  the eigenvalues
$\lambda_{1}\leq\lambda_{2}\leq...\leq\lambda_{i}\leq..., \, \lambda_{i} \rightarrow \infty  ,\,i\rightarrow \infty$ and the   complete orthonormal system of the  eigenvectors.
 Now  note that the operators $V,\,V^{^{\!-\!1}}$ have the same eigenvectors.  Therefore   the operator   $V$ has the  complete orthonormal  system of the  eigenvectors. Recall  that any  complete orthonormal system  forms  a  basis in  the separable Hilbert space. Hence, the complete orthonormal  system of the   eigenvectors of the operator $V$ is
  a basis in the space $\mathfrak{H}.$ Since the operator $A$ is compact and  normal, then in accordance with the well-known theorem we have a fact that there exists an orthonormal system  of the eigenvectors $\{\psi_{i}\}_{1}^{\infty}$ of the operator $A.$ The system is complete in $\overline{\mathrm{R}(A)}$ in the following sense
  $$
 f=\sum\limits_{i=1}^{\infty}\psi_{i}(f,\psi_{i})_{\mathfrak{H}},\;f\in  \overline{\mathrm{R}(A)}.
 $$
  The corresponding system of eigenvalues is such that
  $$
  A\psi_{i}=\lambda_{i}(A)\psi_{i},\;A^{\ast}\psi_{i}=\overline{\lambda_{i}(A)}\psi_{i},\;i\in \mathbb{N}.
  $$
The latter facts give us $A^{\ast}A\psi_{i}=| \lambda_{i}(A) |^{2}\psi_{i}.$ Since the operator $A^{\ast}A$ is   selfadjoint and compact, then it is not hard to prove that  $s_{i}(A)= |\lambda_{i}(A)|$ (see Lemma 3.3 Chapter II \cite{firstab_lit:1Gohberg1965}) Thus, we get
\begin{equation}\label{7a}
s_{i}(A)=|(A \psi_{i},\psi_{i})_{\mathfrak{H}}|=\left(1+\tan^{2} \theta_{i}\right)|\mathrm{Re}(A\psi_{i},\psi_{i})_{\mathfrak{H}}|=
\end{equation}
$$
=
\left(1+\tan^{2} \theta_{i}\right)| (V\psi_{i},\psi_{i})_{\mathfrak{H}}|=\left(1+\tan^{2} \theta_{i}\right) \lambda_{i}(V),
$$
where the sequence $\{\tan^{2} \theta_{i}\}^{\infty}_{1}$  is bounded by virtue of the sectorial property of the operator. Note that the fact $\overline{\mathrm{R}(A)}=\mathfrak{H}$ indicates that $\{\psi_{i}\}_{1}^{\infty}$ is complete  in $\mathfrak{H}.$ It follows that the operators $V$ and $A$ have the same eigenvectors (since the  complete system of the eigenvectors of the operator $V$ is minimal and  at the same time it contains all eigenvectors of the operator $A$).  Therefore, we can claim that all eigenvalues of the operator $V$ are involved in the right-hand side of relation \eqref{7a}.
Taking into account the fact
$
 \lambda_{i}(V) \asymp \lambda_{i}(R_{H}),
$
 we   obtain the following relation
\begin{equation*}
 \sum\limits_{i=1}^{\infty}|s_{i} (A )|^{p}\leq C_{2}\sum\limits_{i=1}^{\infty}|
\lambda_{i} (R_{H} )|^{p},\, p>0.
\end{equation*}
 Using the theorem condition, we have $ A  \in   \,\mathfrak{S}_{p},\,p>1/\kappa.$  Hence   $ A  \in   \,\tilde{\mathfrak{S}}_{\rho},\,\rho\leq1/\kappa.$ At the same time applying the  above reasonings, we get
$$
\frac{\ln^{\kappa} s^{-1}_{i}(A)}{s^{-1}_{i}(A)}\leq C\cdot \frac{\ln^{\kappa} \lambda_{i}(H)}{\lambda_{i}(H)}\leq C\cdot \frac{\alpha_{i}}{i^{\kappa}}.
$$
Using this fact, we obtain the following relation
$$
 \frac{n(r)\ln r}{r^{1/\kappa}}  \rightarrow0,\,r\rightarrow\infty.
$$
Consider the additional condition $  \lambda_{i}(H) =O(i^{\kappa+\varepsilon}),\, \forall\varepsilon>0$ and
 let  $\{\psi_{i}\}_{1}^{\infty}$ still  be the  complete orthonormal system  of the   eigenvectors of the  operator  $V.$     Suppose    $
A\in\mathfrak{S}_{p},\,p\geq1,$ then by virtue of   inequalities  (7.9) Chapter III \cite{firstab_lit:1Gohberg1965}, the fact $\lambda_{i}(V) \asymp \lambda_{i}(R_{H})$
(see Theorem   5 \cite{firstab_lit:1kukushkin2018}),  we get
 $$
\sum\limits_{i=1}^{\infty}|s_{i} (A )|^{p}\geq\sum\limits_{i=1}^{\infty}| (A\varphi_{i},\varphi_{i})_{\mathfrak{H}}|^{p}\geq\sum\limits_{i=1}^{\infty}|{\rm Re}(A\varphi_{i},\varphi_{i})_{\mathfrak{H}}|^{p}=
$$
$$
=\sum\limits_{i=1}^{\infty}| (V \varphi_{i},\varphi_{i})_{\mathfrak{H}}|^{p}=\sum\limits_{i=1}^{\infty}
|\lambda_{i}(V)|^{p}\geq C   \sum\limits_{i=1}^{\infty}  i^{- (\kappa+\varepsilon) p }  .
$$
Therefore   $A\in \mathfrak{S}_{\rho},\, \rho \geq1/\kappa$ since  in the contrary case  the  relation  $p(\kappa+\varepsilon)>1$   does not hold.
The proof is complete.
\end{proof}
Consider the following example.
\begin{ex}\label{E2} Here we would like to produce an example of the sequence   $\{\lambda_{i} \}_{1}^{\infty}$  that satisfies the condition
 $$
 (\ln^{\kappa+1}x)'_{\lambda_{i} }  =o(  i^{-\kappa}   ),\, \kappa\in(0,1],\,
$$
$$
\sum\limits_{n=1}^{\infty}\frac{1}{|\lambda_{n}|^{1/\kappa}}=\infty.
$$
\end{ex}
Consider a sequence $\lambda_{i}=i^{\kappa}\ln^{\kappa} (i+q) \cdot \ln^{\kappa}\ln (i+q),\,q>e^{e}-1,\;i=1,2,...,\,.$  Using the integral test for convergence, we can easily see that the previous series is divergent. At the same time substituting, we get
$$
\frac{\ln^{\kappa}\lambda_{i}}{\lambda_{i}} \leq
 \frac{ C\ln^{\kappa}(i+q)  }{i^{\kappa}\ln^{\kappa} (i+q) \cdot \ln^{\kappa}\ln (i+q)}= \frac{ C }{i^{\kappa}   \cdot \ln^{\kappa}\ln (i+q)},
$$
what gives us the fulfilment of the  condition.\\

Bellow, we produce an auxiliary technique to study the central problem of the paper. The estimates for the Fredholm  Determinant were studied by Lidskii  in the paper \cite{firstab_lit:1Lidskii} and gave a main tool in questions related to   the  estimation of the contour integrals. We have slightly improved results by Lidskii having involved the auxiliary function $\beta$ and  obtaining in this way more accurate  results.\\

\noindent{\bf  Estimates for the Fredholm  Determinant}\\

In this section we produce an adopted version of the propositions given in the paper \cite{firstab_lit:1Lidskii}, we consider a case when a compact operator  belongs to the class $\tilde{\mathfrak{S}}_{\rho}.$  Having taken into account the facts considered in the previous subsection, we can reformulate Lemma 2 \cite{firstab_lit:1Lidskii} in the refined form.
\begin{lem}\label{L5}   Assume that a compact operator $B$ satisfies   the  condition   $B\in \tilde{\mathfrak{S}}_{\rho},$       then
for arbitrary numbers  $R,\delta$   such that $R>0,\,0<\delta<1,$ there exists a  circle $|\lambda|=\tilde{R},\,(1-\delta)R<\tilde{R}<R,$ so that the following estimate holds
$$
\|(I-\lambda B )^{-1}\|_{\mathfrak{H}}\leq e^{\gamma(|\lambda|)|\lambda|^{\varrho}}|\lambda|^{m},\,|\lambda|=\tilde{R},\,m=[\varrho],\,\varrho\geq\rho,
$$
where
$$
\gamma(|\lambda|)= \beta ( |\lambda|^{m+1})  +C \beta(|C  \lambda| ^{m+1}),\;\beta(r )= r^{ -\frac{\varrho}{m+1} }\left(\int\limits_{0}^{r}\frac{n_{B^{m+1}}(t)dt}{t }+
r \int\limits_{r}^{\infty}\frac{n_{B^{m+1}}(t)dt}{t^{ 2  }}\right).
$$
\end{lem}
\begin{proof} We consider the case $\varrho=\rho,$ the reasonings corresponding to the case $\varrho>\rho$  can be fulfilled in accordance with the same scheme but much simpler and left to  the reader.  Using  the definition, we have  $B\in \mathfrak{S}_{\rho+\varepsilon},\,\varepsilon>0.$
By direct calculation we get
\begin{equation}\label{8a}
(I-\lambda^{m+1}B^{m+1})^{-1}(I+\lambda B+\lambda^{2} B^{2}+...+\lambda^{m}B^{m})=(I-\lambda B )^{-1}.
\end{equation}
In accordance with Lemma 3 \cite{firstab_lit:1Lidskii}, for sufficiently small $\varepsilon>0,$ we have
$$
\sum\limits_{i=1}^{\infty}\lambda^{\frac{\rho+\varepsilon}{m+1}}_{i}( \tilde{B} )\leq \sum\limits_{i=1}^{\infty}\lambda^{ \rho+\varepsilon }_{i}( B )<\infty,
$$
where $\tilde{B}:=(B^{\ast m+1}B^{m+1})^{1/2}.$       Applying  inequality  (1.27) \cite[p.10]{firstab_lit:1Lidskii} (since $  \rho  /(m+1)<1$), using Lemma \ref{L2},   we get
$$
\|\Delta_{B^{m+1}}(\lambda^{m+1})(I-\lambda^{m+1} B^{m+1})^{-1}\|_{\mathfrak{H}}\leq C\prod\limits_{i=1}^{\infty}\{1+|\lambda^{m+1} s_{i}( B^{m+1} )|\}\leq Ce^{ \beta (r^{m+1})r^{\rho }},
$$
where $\Delta_{B^{m+1}}(\lambda^{m+1})$ is a Fredholm  determinant of the operator $B^{m+1}$  (see \cite[p.8]{firstab_lit:1Lidskii}).
In accordance with Theorem 11 \cite[p.33]{firstab_lit:Eb. Levin}, we have
$$
\Delta_{B^{m+1}}(\lambda^{m+1})\geq e^{-(2+\ln\{12e/\delta\})\ln\xi_{m}},\;\xi_{m}= \!\!\!\max\limits_{\psi\in[0,2\pi /(m+1)]}\{\Delta_{B^{m+1}}([2e \tilde{R}e^{i\psi}]^{m+1})\},
$$
where $R,\delta$ arbitrary numbers such that $R>0,\,0<\delta<1,$ the values of $\lambda$ belong to the  circle $|\lambda|=\tilde{R},$ which radius   is defined by $R,\delta$  and satisfy the condition $(1-\delta)R<\tilde{R}<R.$ Note that in accordance with the estimate (1.21) \cite[p.10]{firstab_lit:1Lidskii}, we have
$$
\Delta_{B^{m+1}}(\lambda)\leq C\prod\limits_{i=1}^{\infty}\{1+|\lambda s_{i}( B^{m+1} )|\}.
$$
Therefore, applying  Lemma    \ref{L2},  we get
$
\xi_{m}\leq e^{ \beta ([2e \tilde{R }]^{m+1})( 2e \tilde{R}  )^{\rho}}.
$
Consider relation \eqref{8a}, we have the following estimate
$$
 \|(I-\lambda B )^{-1}\|_{\mathfrak{H}}\leq\|(I-\lambda^{m+1}B^{m+1})^{-1}\|_{\mathfrak{H}} \cdot\|(I+\lambda B+\lambda^{2} B^{2}+...+\lambda^{m}B^{m})\|_{\mathfrak{H}}\leq
$$
$$
\leq\|(I-\lambda^{m+1}B^{m+1})^{-1}\|_{\mathfrak{H}} \cdot \frac{|\lambda|^{m+1}\|B\|^{m+1}-1}{|\lambda|\cdot\|B\|-1}.
$$
We can easily see   that to obtain the desired result  it suffices to estimate the term $\|(I-\lambda^{m+1}B^{m+1})^{-1}\|_{\mathfrak{H}}.$ Using the obtained estimates, we have
$$
\|(I-\lambda^{m+1}B^{m+1})^{-1}\|_{\mathfrak{H}}\leq  e^{\gamma (  |\lambda|) |\lambda| ^{\rho}},\,|\lambda|=\tilde{R},
$$
where $\gamma (|\lambda|)= \beta ( |\lambda|^{m+1})  +(2+\ln\{12e/\delta\}) \beta_{m}(|2e  \lambda| ^{m+1}) (2e) ^{\rho}.$   Thus, we obtain the desired result.
\end{proof}

\noindent{\bf Abel-Lidsky summarizing the series}\\

 In this subsection, we reformulate  results obtained by Lidskii \cite{firstab_lit:1Lidskii} in a more  convenient  form applicable to the reasonings of this paper.   However,  let us begin our narrative.    In accordance with the Hilbert theorem
  (see \cite{firstab_lit:Riesz1955}, \cite[p.32]{firstab_lit:1Gohberg1965})   the spectrum of an arbitrary  compact operator $B$  consists of the so called normal eigenvalues it gives us the opportunity to consider a decomposition
\begin{equation}\label{9a}
 \mathfrak{H}=\mathfrak{N}_{q}\stackrel{\cdot}{+}  \mathfrak{M}_{q},
 \end{equation}
where both  summands are   invariant subspaces regarding the operator $B,$  the first one is  a finite dimensional root subspace corresponding to the eigenvalue $\mu_{q}$ and the second one is a subspace  wherein the operator  $B-\mu_{q} I$ is invertible.  Let $n_{q}$ is a dimension of $\mathfrak{N}_{q}$ and let $B_{q}$ is the operator induced in $\mathfrak{N}_{q}.$ We can choose a basis (Jordan basis) in $\mathfrak{N}_{q}$ that consists of Jordan chains of eigenvectors and root vectors  of the operator $B_{q}.$  Each chain $e_{q_{\xi}},e_{q_{\xi}+1},...,e_{q_{\xi}+k},$ where $e_{q_{\xi}},\,\xi\in \mathbb{N}$ are the eigenvectors  corresponding   to the  eigenvalue $\mu_{q}$  and other terms are root vectors,   can be transformed by the operator $A$ according with  the following formulas
\begin{equation}\label{10a}
Be_{q_{\xi}}=\mu_{q}e_{q_{\xi}},\;Be_{q_{\xi}+1}=\mu_{q}e_{q_{\xi}+1}+e_{q_{\xi}},...,Be_{q_{\xi}+k}=\mu_{q}e_{q_{\xi}+k}+e_{q_{\xi}+k-1}.
\end{equation}
Considering the sequence $\{\mu_{i}\}_{1}^{\infty}$ of the eigenvalues of the operator $B$ and choosing a  Jordan basis in each corresponding  space $\mathfrak{N}_{i}$ we can arrange a system of vectors $\{e_{k}\}_{1}^{\infty}$ which we will call a system of the root vectors or following  Lidskii  a system of the major vectors of the operator $A.$
Let $e_{1},e_{2},...,e_{n_{i}}$ be the Jordan basis in the subspace $\mathfrak{N}_{i},$ then in accordance with    Lidskii \cite{firstab_lit:1Lidskii}   there exists a  corresponding biorthogonal basis $g_{1},g_{2},...,g_{n_{i}}$ in the space $\mathfrak{M}_{i}^{\perp}$ (see \cite[p.14]{firstab_lit:1Lidskii}), note that in accordance with our clarification $\mathfrak{M}_{i}^{\perp}=\mathfrak{N}_{i}.$  Moreover the set $\{ g_{k}\}_{1}^{n_{i}}$ consists of the Jordan chains of the operator $B^{\ast}$ which correspond to the Jordan chains  \eqref{10a} due to the following formula
$$
B^{\ast}g_{q_{\xi}+k}=\overline{\mu_{q}}g_{q_{\xi}+k},\;B^{\ast}g_{q_{\xi}+k-1}=\overline{\mu_{q}}g_{q_{\xi}+k-1}+g_{q_{\xi}+k},...,
B^{\ast}g_{q_{\xi}}=\overline{\mu_{q}}g_{q_{\xi}}+g_{q_{\xi}+1}.
$$
 Let us show that   $\mathfrak{N}_{i}\subset \mathfrak{M}_{j},\,i\neq j$ for this purpose note that in accordance with the  representation  $P_{\mu_{i}}\mathfrak{H}=\mathfrak{N}_{i}$ and the property $P_{\mu_{i}}P_{\mu_{j}}=0,\,i\neq j,$ where $P_{\mu_{i}}$  is a Riesz projector (integral) corresponding to the eigenvalue $\mu_{i}$ (see \cite{firstab_lit:1Gohberg1965} Chapter I \S 1.3), we have an orthogonal  decomposition
 $
 \mathfrak{H}=\mathfrak{N}_{i} \stackrel{\cdot}{+} \mathfrak{N}_{j}\stackrel{\cdot}{+} \mathfrak{M}_{ij},
 $
where $\mathfrak{M}_{ij}=(I-P_{\mathfrak{N}_{i} \stackrel{\cdot}{+} \mathfrak{N}_{j}})\mathfrak{H}.$
On the other hand in accordance with \cite{firstab_lit:1Gohberg1965} Chapter I \S 2.1   we can claim that the following orthogonal  decomposition is unique
 $$
 \mathfrak{H}= \mathfrak{N}_{j}\stackrel{\cdot}{+} \mathfrak{M}_{j},
 $$
hence  we have an orthogonal sum
$
 \mathfrak{M}_{j} = \mathfrak{N}_{i}\stackrel{\cdot}{+} \mathfrak{M}_{ij},
 $
what proves the desired result.
 Taking into account relation  \eqref{9a}, we conclude that  the set  $g_{1},g_{2},...,g_{n_{i}},\,i\neq j$  is orthogonal to the set $e_{1},e_{2},...,e_{n_{j}}.$  Gathering the sets $g_{1},g_{2},...,g_{n_{i}},\,i=1,2,...,$ we can obviously create a biorthogonal system $\{g_{i}\}_{1}^{\infty}$ with respect to the system of the major vectors of the operator $B.$ It is rather reasonable to call it as  a system of the major vectors of the operator $B^{\ast}.$ Note that if an element $f\in\mathfrak{H}$ allows a decomposition in the strong sense
$$
f=\sum\limits_{n=1}^{\infty}e_{n}c_{n},\,c_{n}\in \mathbb{C},
$$
then by virtue of  the biorthogonal  system existing, we can claim that such a representation is unique. Further, let us come to the previously made  agrement that the vectors in each Jourdan chain are arranged in the same order as in \eqref{10a} i.e.  at the first place there stands an eigenvector. It is clear that under such an assumption  we have
$$
c_{q_{\xi}+i}=\frac{(f,g_{q_{\xi}+k-i})}{(e_{q_{\xi}+i},g_{q_{\xi}+k-i})},\,0\leq i\leq k(q_{\xi}),
$$
where $k(q_{\xi})+1$ is a number of elements in the $q_{\xi}$-th Jourdan chain. In particular, if the vector $e_{q_{\xi}}$ is included to the major system solo, there does not exist a root vector corresponding to the same eigenvalue, then
$$
c_{q_{\xi}}=\frac{(f,g_{q_{\xi}})}{(e_{q_{\xi}},g_{q_{\xi}})}.
$$
Note that in accordance with the property of the biorthogonal sequences, we can expect that the denominators equal to one in the previous two relations.
Consider a formal series corresponding to a decomposition on the major vectors of the operator $B$
$$
f\sim\sum\limits_{n=1}^{\infty}e_{n}c_{n},
$$
where each number $n$ corresponds to a number $q_{\xi}+i$  (thus, the coefficients $c_{n}$ are defined in accordance with the above and  numerated in a simplest way). Consider a set of the polynomials with respect to a real parameter $t$
$$
P^{\alpha}_{m}(\zeta^{-1},t)=\frac{e^{t\zeta^{-\alpha}}}{m!}\frac{d^{m}}{d\zeta^{m}}\,e^{-t\zeta^{-\alpha}},\,\alpha>0,\,m=1,2,...,\,.
$$
Consider a series
\begin{equation}\label{11a}
\sum\limits_{n=1}^{\infty}c_{n}(t)e_{n},
\end{equation}
where the coefficients $c_{n}(t)$ are defined  in accordance with the correspondence between the indexes  $n$ and $q_{\xi}+i$ in the following way
\begin{equation}\label{11x}
c_{q_{\xi}+i}(t)=e^{-\lambda^{\alpha}_{q}t}\sum\limits_{m=0}^{k-i}P_{m}^{\alpha}(\lambda_{q_{\xi}},t)c_{q_{\xi}+i+m},\,i=0,1,2,...,k,
\end{equation}
here $\lambda_{q}=1/\mu_{q}$ is a characteristic number corresponding to $e_{q_{\xi}}.$
It is clear  that in any case, we have
$
c_{n}(t)\rightarrow c_{n},\,t\rightarrow0
$
(it can be established by direct calculations).
In accordance with the definition given in \cite[p.17]{firstab_lit:1Lidskii} we will say that series \eqref{11a} converges  to the element $f$ in the sense $(B,\lambda,\alpha),$ if there exists a sequence of the natural numbers $\{N_{j}\}_{1}^{\infty}$ such that
$$
f=\lim\limits_{t\rightarrow+0}\lim\limits_{j\rightarrow\infty}\sum\limits_{n=1}^{N_{j}}c_{n}(t)e_{n}.
$$
Note that   sums of the latter relation forms a subsequence of the partial sums of the series \eqref{11a}.

To establish the main results  we need the following lemmas by Lidskii. Note that in spite of the fact that we have rewritten the lemmas in the refined form the proof has not been changed  and can be found in the paper \cite{firstab_lit:1Lidskii}.
  Further, considering an arbitrary  compact operator $B: \mathfrak{H}\rightarrow \mathfrak{H}$ such that
$
\Theta(B)\subset \mathfrak{L}_{0}(\theta),\,-\pi<\theta<\pi,
$
we put the following contour   in correspondence to the operator
\begin{equation}\label{11f}
\gamma(B):=\left\{\lambda:\;|\lambda|=r>0,\,|\mathrm{arg} \lambda|\leq \theta+\varepsilon\right\}\cup\left\{\lambda:\;|\lambda|>r,\; |\mathrm{arg} \lambda|=\theta+\varepsilon\right\},
\end{equation}
where $\varepsilon>0$ is an arbitrary small number, the number $r$ is chosen so that the operator  $ (I-\lambda B)^{-1} $ is regular within the corresponding closed circle. Here we should note that the compactness property of $B$ gives us the fact   $(I-\lambda B)^{-1}\in \mathcal{B}(\mathfrak{H}),\,\lambda \in \mathbb{C}\setminus \mathrm{ int} \gamma\, (B).$ It can be proved easily if we note that in accordance with the Corollary 3.3 \cite[p.268]{firstab_lit:kato1980}, we have $\mathrm{P}(B)\subset \mathbb{C}\setminus \overline{\Theta(B)}.$

\begin{lem} \label{L6} Assume that $B$ is a compact  operator,  $\Theta(B)\subset \mathfrak{L}_{0}(\theta),\,-\pi<\theta<\pi,$ then on each ray $\zeta$ containing the point zero and not belonging to the sector $\mathfrak{L}_{0}(\theta)$ as well as the  real axis, we have
$$
\|(I-\lambda B)^{-1}\|\leq \frac{1}{\sin\varphi},\,\lambda\in \zeta,
$$
where $\,\varphi = \min \{|\mathrm{arg}\zeta -\theta|,|\mathrm{arg}\zeta +\theta|\}.$
\end{lem}

\begin{lem}\label{L7} Assume that the operator $B$ satisfies conditions of Lemma \ref{L6}, $f\in \mathrm{R}(B),$ then
$$
\lim\limits_{t\rightarrow+0}\int\limits_{\gamma(B)}e^{-\lambda^{\alpha}t}B(I-\lambda B)^{-1}fd\lambda=f,\,\alpha>0.
$$
\end{lem}
\begin{lem}\label{L8} Assume that   $B$ is a compact operator, then  in the pole $\lambda_{q}$ of the operator  $(I-\lambda B)^{-1},$ the residue of the vector  function $e^{-\lambda^{\alpha}t}B(I-\lambda B)^{-1}\!f,\,(f\in \mathfrak{H}),\,\alpha>0$  equals to
$$
-\sum\limits_{\xi=1}^{m(q)}\sum\limits_{i=0}^{k(q_{\xi})}e_{q_{\xi}+i}c_{q_{\xi}+i}(t),
$$
where $m(q)$ is a geometrical multiplicity of the $q$-th eigenvalue,  $k(q_{\xi})+1$ is a number of elements in the $q_{\xi}$-th Jourdan chain, the coefficients $c_{q_{\xi}+i}(t)$ are defined in accordance with  formula \eqref{11x}.
\end{lem}

\section{Main results}
 In this section,  considering the class  $\tilde{\mathfrak{S}}_{\alpha}$ under   additional assumptions, we improve results obtained by Lidskii  \cite{firstab_lit:1Lidskii}. As an application we consider differential equations in the Hilbert space. We should stress that a significant refinement takes place in comparison with the reasonings \cite{firstab_lit:1Lidskii}.
  We consider the operator classes under the point of view made in the latter section. Firstly, we consider a general statement
with the made refinement related to the involved notion of the convergence exponent. Secondly, having formulated  conditions in terms of the operator order,
we produce an example establishing the fact in accordance with which the contours may be chosen in a concrete way, under the assumption $\rho=\alpha,$    what  provides a peculiar  validity of the  statement.
 Finally, we  consider   applications  to the differential equations in the Hilbert space.
For convenience,   we will use   auxiliary denotations
$$
I=\sum\limits_{\nu=0}^{\infty}I_{\nu};\;J^{+}=\sum\limits_{\nu=0}^{\infty}J^{+}_{\nu};\;J^{-}=\sum\limits_{\nu=0}^{\infty}J^{-}_{\nu}.
$$
The structure of   the  proof of the  following theorem   completely belongs to Lidskii. However, we produce the proof since we make a refinement corresponding to consideration of the case when a convergence exponent does not equals the index of the Schatten-von Neumann  class.
\begin{teo}\label{T2} Assume that $B$ is a compact operator, $\Theta(B) \subset   \mathfrak{L}_{0}(\theta),\,\theta<\min\{\pi/2\alpha,\pi \},\;B\in\tilde{\mathfrak{S}}_{\rho},\,0<\rho\leq\alpha.$   Moreover  in the case $B  \in \tilde{\mathfrak{S}}_{\rho}\setminus  \mathfrak{S}_{\rho}$ the additional condition holds
\begin{equation}\label{12d}
   \frac{n_{B^{m+1}}(r^{m+1})}{r^{\rho} }\rightarrow 0,\, m=[\rho].
\end{equation}
Then a sequence of natural numbers $\{N_{\nu}\}_{0}^{\infty}$ can be chosen so that
$$
\frac{1}{2\pi i}\int\limits_{\gamma(B)}e^{-\lambda^{\alpha}t}B(I-\lambda B)^{-1}f d \lambda =   \sum\limits_{\nu=0}^{\infty} \sum\limits_{q=N_{\nu}+1}^{N_{\nu+1}}\sum\limits_{\xi=1}^{m(q)}\sum\limits_{i=0}^{k(q_{\xi})}e_{q_{\xi}+i}c_{q_{\xi}+i}(t) ,
$$
  moreover
\begin{equation}\label{12a}
  \sum\limits_{\nu=0}^{\infty}\left\|\sum\limits_{q=N_{\nu}+1}^{N_{\nu+1}}\sum\limits_{\xi=1}^{m(q)}\sum\limits_{i=0}^{k(q_{\xi})}e_{q_{\xi}+i}c_{q_{\xi}+i}(t)\right\|_{\mathfrak{H}}<\infty.
\end{equation}
\end{teo}

\begin{proof}
Consider a contour $\gamma(B).$ Having fixed $R>0,0<\delta<1,$ so that $R(1-\delta)=r,$ consider a monotonically increasing sequence $\{R_{\nu}\}_{0}^{\infty},\,R_{\nu}=R(1-\delta)^{-\nu+1}.$   Using Lemma \ref{5}, we get
$$
\|(I-\lambda B )^{-1}\|_{\mathfrak{H}}\leq e^{\gamma (|\lambda|)|\lambda|^{\rho}}|\lambda|^{m},\,m=[\rho],\,|\lambda|=\tilde{R}_{\nu},\;R_{\nu}<\tilde{R}_{\nu}<R_{\nu+1},
$$
where the function  $\gamma (r)$ is defined in Lemma \ref{5},
$$
\beta(r )= r^{ -\frac{\rho}{m+1} }\left(\int\limits_{0}^{r}\frac{n_{B^{m+1}}(t)}{t }dt+
r \int\limits_{r}^{\infty}\frac{n_{B^{m+1}}(t)}{t^{ 2  }}dt\right).
$$
Denote by $\gamma_{\nu}$ a bound of the intersection of the ring $\tilde{R}_{\nu}<|\lambda|<\tilde{R}_{\nu+1}$ with the interior of the contour $\gamma(B),$ denote by $N_{\nu}$ a number of poles being   contained  in the set $\mathrm{int }\,\gamma (B) \,\cap \{\lambda:\,r<|\lambda|<\tilde{R}_{\nu} \}.$ In accordance with Lemma \ref{L8},  we get
 $$
 \frac{1}{2\pi i}\int\limits_{\gamma_{\nu}}e^{-\lambda^{\alpha}t}B(I-\lambda B)^{-1}f d \lambda =\sum\limits_{q=N_{\nu}+1}^{N_{\nu+1}}\sum\limits_{\xi=1}^{m(q)}\sum\limits_{i=0}^{k(q_{\xi})}e_{q_{\xi}+i}c_{q_{\xi}+i}(t).
$$
Let us estimate the above integral, for this purpose split the contour $\gamma_{\nu}$ on  terms $\tilde{\gamma}_{ \nu  }:=\{\lambda:\,|\lambda|=\tilde{R}_{\nu},\,|\mathrm{arg} \lambda |<\theta +\varepsilon\},\, \gamma_{\nu_{+}}:=
\{\lambda:\,\tilde{R}_{\nu}<|\lambda|<\tilde{R}_{\nu+1},\, \mathrm{arg} \lambda  =\theta +\varepsilon\},\,\gamma_{\nu_{-}}:=
\{\lambda:\,\tilde{R}_{\nu}<|\lambda|<\tilde{R}_{\nu+1},\, \mathrm{arg} \lambda  =-\theta -\varepsilon\}.$    In accordance with Lemma \ref{L5}, we have
$$
 I_{  \nu  }: =\left\|\,\int\limits_{\tilde{\gamma}_{ \nu }}e^{-\lambda^{\alpha}t}B(I-\lambda B)^{-1}f d \lambda\,\right\|_{\mathfrak{H}}\leq \,\int\limits_{\tilde{\gamma}_{ \nu }}e^{-\lambda^{\alpha}t}\left\|B(I-\lambda B)^{-1}f \right\|_{\mathfrak{H}} |d \lambda|\leq
$$
 $$
 \leq e^{\gamma (|\lambda|)|\lambda|^{\rho} }|\lambda|^{m+1}\int\limits_{-\theta-\varepsilon}^{\theta+\varepsilon} e^{- t \mathrm{Re}\lambda^{\alpha}} d \,\mathrm{arg} \lambda,\,|\lambda|=\tilde{R}_{\nu}.
$$
Using the theorem conditions,    we get     $\,|\mathrm{arg} \lambda |<\pi/2\alpha,\,\lambda\in \tilde{\gamma}_{ \nu },\,\nu=0,1,2,...\,.$  It follows that
$$
\mathrm{Re }\lambda^{\alpha}\geq |\lambda|^{\alpha} \cos \left[(\pi/2\alpha-\delta)\alpha\right]= |\lambda|^{\alpha} \sin     \alpha \delta,
$$
 where $\delta$ is a sufficiently small number.
Thus,  we get
$$
I_{ \nu } \leq     Ce^{\gamma (|\lambda|)|\lambda|^{\rho}-t|\lambda|^{\alpha} \sin     \alpha \delta}|\lambda|^{m+1}=
   e^{|\lambda|^{\rho}[\gamma (|\lambda|)-t|\lambda|^{\alpha-\rho} \sin     \alpha \delta]}|\lambda|^{m+1}
 ,\,m=[\rho],\,|\lambda|=\tilde{R}_{\nu}.
$$
Let us show that for a fixed $t$ and  a sufficiently large $|\lambda|,$ we have  $ \gamma (|\lambda|)-t|\lambda|^{\alpha-\rho} \sin     \alpha \delta  < 0.$
It follows  directly from Lemma \ref{L2} in the case when  $B\in \mathfrak{S}_{\rho}$ as well as in the case $B  \in \tilde{\mathfrak{S}}_{\rho}\setminus  \mathfrak{S}_{\rho}$ but here we should involve the additional condition \eqref{12d}.
 Therefore,   the   series   $I$ converges.
Using the analogous estimates, applying  Lemma \ref{L6}, we get
$$
 J^{+}_{\nu}: =\left\|\,\int\limits_{\gamma_{\nu_{+}}}e^{-\lambda^{\alpha}t}B(I-\lambda B)^{-1}f d \lambda\,\right\|_{\mathfrak{H}}\leq  C\|f\|_{\mathfrak{H}}  \int\limits_{R_{\nu}}^{R_{\nu+1}} |e^{- t \lambda^{\alpha}}|  |d   \lambda|\leq C  e^{-tR_{\nu} ^{\alpha} \sin     \alpha\,\varepsilon}   \int\limits_{R_{\nu}}^{R_{\nu+1}}   |d   \lambda|=
 $$
 $$
 = Ce^{-tR_{\nu} ^{\alpha} \sin     \alpha\,\varepsilon}\{R_{\nu+1}-R_{\nu} \};
$$
$$
 J^{-}_{\nu}: =\left\|\,\int\limits_{\gamma_{\nu_{-}}}e^{-\lambda^{\alpha}t}B(I-\lambda B)^{-1}f d \lambda\,\right\|_{\mathfrak{H}}\leq   Ce^{-tR_{\nu} ^{\alpha} \sin     \alpha\,\varepsilon}\{R_{\nu+1}-R_{\nu} \}.
$$
Therefore, the series $ J^{+}, J^{-}$ are convergent. Thus, we obtain relation  \eqref{12a}, from what follows the rest part of the theorem claim.
\end{proof}

\noindent{\bf  Sequence of  power type contours}\\

It is remarkable that we can choose a sequence of contours in various ways. For instance, a sequence of  contours of the exponential  type was considered in the paper \cite{firstab_lit:1Lidskii}.
In this paragraph, we produce an application of the  previous section results, we study a concrete operator class for which it is possible to choose a sequence of   contours of the power type. At the same time  having involved an additional condition we can spread the principal result of the paragraph  on a wider operator class. Note that using   condition  $\mathrm{H}2,$ it is not hard to prove that
$
\mathrm{Re}(\tilde{W}f,f)_{\mathfrak{H}} -k|\mathrm{Im}(\tilde{W}f,f)_{\mathfrak{H}}|\geq (C_{2} -k C_{1})\|f\|^{2}_{\mathfrak{H}_{+}},\,k>0,
$
from what follows a fact  $\Theta(A)\subset \mathfrak{L}_{0}(\theta),\,\theta=\arctan (C_{1}/C_{2}).$ In general, the last relation gives us a range  of the semi-angle $\pi/4<\theta<\pi/2,$  thus the     conditions $\mathrm{H}1,\mathrm{H}2$ are not sufficient to guaranty a value of the semi-angle less than $\pi/4.$ However, we should remark that  some relevant   results can be obtained in the very case of  sufficiently small values of the semi-angle, this gives us a motivation to consider a more specific  additional  assumption \\

 \noindent  $( \mathrm{H3} )  \;  |\mathrm{Im}(\tilde{L}f,g)_{\mathfrak{H}}|\! \leq \! C_{3}\|f\|_{\mathfrak{H_{+}}}\|g\|_{\mathfrak{H}} ,\,
     ,\,f,g \in  \mathfrak{M},\;  C_{3}>0.
$\\

 \noindent In this case, we have
$
\mathrm{Re}(\tilde{W}f,f)_{\mathfrak{H}} -k|\mathrm{Im}(\tilde{W}f,f)_{\mathfrak{H}}|\geq
 C_{2}\|f\|_{\mathfrak{H_{+}}}  -k C_{3}\left\{\varepsilon\|f\|^{2}_{\mathfrak{H_{+}}}/2+ \|f\|^{2}_{\mathfrak{H}}/2\varepsilon \right\}\geq
(C_{2}-k\varepsilon C_{3})|f\|^{2}_{\mathfrak{H_{+}}}/2+(C_{2}/C_{0}-k  C_{3}/\varepsilon)|f\|^{2}_{\mathfrak{H }}/2,\,k>0.
 $
Thus, choosing $\varepsilon=C_{2}/kC_{3},$ we get $\Theta(\tilde{W})\subset \mathfrak{L}_{\iota}(\theta_{\iota}),$ where $\iota=C_{2}/2C_{0}-(k  C_{3})^{2}/2C_{2},\,\theta_{\iota}=\arctan (1/k).$
This relation   guarantees  that  we can choose a sufficiently small value of the semi-angle $\theta_{\iota}.$  We
put the following contour in correspondence to an operator $L$ satisfying the additional condition $\mathrm{H}3$
$$
\Gamma(A):= \mathrm{Fr}\left\{ \left( \mathfrak{L}_{0}(\theta_{0}+\varepsilon) \cap \mathfrak{L}_{\iota}(\theta_{\iota}+\varepsilon) \right)  \setminus
\mathfrak{C}_{r}\right\},\,\iota<0,\;\mathfrak{C}_{r}:=\left\{\lambda:\;|\lambda|<r,\,|\mathrm{arg} \lambda|\leq \theta_{0} \right\},
$$
where $r$ is chosen so that the   operator $ (I-\lambda A)^{-1}  $ is regular within the corresponding  closed circle,  $\varepsilon>0$ is    sufficiently small.

\begin{lem} \label{L9} Assume that condition $\mathrm{H}3 $ holds, then
$$
\|(I-\lambda A)^{-1}\|_{\mathfrak{H}}\leq C,\,\lambda\in \mathrm{Fr}\left\{   \mathfrak{L}_{0}(\theta_{0}+\varepsilon)\cap \mathfrak{L}_{\iota}(\theta_{\iota}+\varepsilon)\right\},\,\iota<0,
$$
where $\iota=C_{2}/2C_{0}-(k  C_{3})^{2}/2C_{2},\,\theta_{\iota}=\arctan (1/k),\,\varepsilon>0$ is an arbitrary small   number.
 \end{lem}
\begin{proof} Firstly, we should note that in accordance with condition $\mathrm{H}3,$  for an arbitrary large value $k,$ we have
  $\Theta(\tilde{W})\subset \mathfrak{L}_{\iota}(\theta_{\iota}),$ where $\iota=C_{2}/2C_{0}-(k  C_{3})^{2}/2C_{2},\,\theta_{\iota}=\arctan (1/k).$
 Hence $\Theta(\tilde{W})\subset  \mathfrak{L}_{0}(\theta_{0} )\cap \mathfrak{L}_{\iota}(\theta_{\iota} ),$ where $\iota$ is arbitrary negotive. Therefore
$\Theta(A)\subset  \mathfrak{L}_{0}(\theta_{0} )\cap \mathfrak{L}_{\iota}(\theta_{\iota} ),$ it can be verified directly due to the geometrical methods. Note that by virtue of the Lemma 4 \cite{firstab_lit:1Lidskii}, we have
$
\|(I-\lambda A)^{-1}\|_{\mathfrak{H}}\leq C,\,\lambda\in \mathrm{Fr}\left\{   \mathfrak{L}_{0}(\theta_{0}+\varepsilon)\right\}.
$
Thus to obtain the desired result it suffices to prove that
$
\|(I-\lambda A)^{-1}\|_{\mathfrak{H}}\leq C,\,\lambda\in \mathrm{Fr}\left\{   \mathfrak{L}_{\iota}(\theta_{\iota}+\varepsilon)\right\},\,\mathrm{Re}\lambda\geq 0.
$
Note that in this case  $\lambda \in \mathrm{P}(\tilde{W})$  and  we have a chain of   reasonings $\forall f\in \mathfrak{H}:\;(\tilde{W}-\lambda I)^{-1}f=h\in \mathrm{D}(\tilde{W});$
$\; (\tilde{W}-\lambda I)h =f;\;(f,h)_{\mathfrak{H}}=(\tilde{W}h,h)_{\mathfrak{H}}-\lambda \|h\|_{\mathfrak{H}}^{2}.$ Using the latter relation, we get
$
 |  (f,h)_{\mathfrak{H}}| /\|h\|_{\mathfrak{H}}^{2}   =  |  (\tilde{W}h,h)_{\mathfrak{H}}/\|h\|_{\mathfrak{H}}^{2} -\lambda  |\geq |\lambda-\iota| \sin \varepsilon.
$
Therefore, using the Cauchy-Schwartz inequality, we get
$$
 \|(\tilde{W}-\lambda I)^{-1}f\|_{\mathfrak{H}}\leq \frac{1}{|\lambda-\iota|\sin  \varepsilon }\cdot\|f\|_{\mathfrak{H}},\;f\in \mathfrak{H}.
$$
Taking into account the fact
$
(\tilde{W}-\lambda I)^{-1} =(I-\lambda A )^{-1}A= \{(I-\lambda A )^{-1}-I \}/\lambda,
$
we get
$
\| (I-\lambda A )^{-1} \|_{\mathfrak{H}} -1\leq\| (I-\lambda A )^{-1}-I  \|_{\mathfrak{H}}\leq  |\lambda|/ |\lambda-\iota|\sin  \varepsilon    ,
$
from what follows the desired result.
 \end{proof}

\begin{lem}\label{L10} Assume that  the condition $\mathrm{H}3$ holds, $f\in \mathrm{R}(A),$ then
$$
\lim\limits_{t\rightarrow+0}\int\limits_{\Gamma(A)}e^{-\lambda^{\alpha}t}A(I-\lambda A)^{-1}fd\lambda=f,\,\alpha>0.
$$
\end{lem}
\begin{proof} The proof is analogous to the proof of the Lemma 5 \cite{firstab_lit:1Lidskii} and  the only difference is in the following, we should use Lemma \ref{L9} instead of Lemma \ref{L6}. 
\end{proof}

The  theorems given bellow   are formulated under the assumption that  either  the condition $\Theta(A) \subset   \mathfrak{L}_{0}(\theta),\,\theta< \pi/2\alpha$ or   condition $\mathrm{H}3$ holds. In accordance with  such an alternative,   we put in correspondence $\varpi:=\gamma(A)$ or $\varpi:=\Gamma(A)$ respectively.  
The following  theorem  is similar to the result  \cite{firstab_lit:1Lidskii}, but formulated in   terms of the operator order. Although  the principal clarification $\alpha=\rho$ has not been obtained it can be interesting by virtue of the different  way of chousing a sequence of contours.

\begin{teo}\label{T3}
Assume that  the operator $\tilde{W}$ satisfies the condition   $\alpha>2/\mu,\,\mu\in (0,1]$ and $\alpha> 1,\,\mu\in (1,\infty).$  
Then a sequence of   natural  numbers  $\{N_{\nu}\}_{0}^{\infty}$ can be chosen   so that
$$
 \frac{1}{2\pi i}\int\limits_{\varpi}e^{-\lambda^{\alpha}t}A(I-\lambda A)^{-1}f d \lambda
=  \sum\limits_{\nu=0}^{\infty}  \sum\limits_{q=N_{\nu}+1}^{N_{\nu+1}}\sum\limits_{\xi=1}^{m(q)}\sum\limits_{i=0}^{k(q_{\xi})}e_{q_{\xi}+i}c_{q_{\xi}+i}(t),
$$
where
\begin{equation}\label{13e}
  \sum\limits_{\nu=0}^{\infty}\left\|\sum\limits_{q=N_{\nu}+1}^{N_{\nu+1}}\sum\limits_{\xi=1}^{m(q)}\sum\limits_{i=0}^{k(q_{\xi})}e_{q_{\xi}+i}c_{q_{\xi}+i}(t)\right\|_{\mathfrak{H}}<\infty,
\end{equation}
the following relation holds  for the eigenvalues
$$
|\lambda_{N_{\nu}+k}|-|\lambda_{N_{\nu}+k-1}|\leq C |\lambda_{N_{\nu}+k}|^{1-1/\tau},\;
 k=2,3,...,N_{\nu+1}-N_{\nu},\;0<\tau<\mu.
$$
\end{teo}
\begin{proof}
  In accordance with Theorem  \ref{T1},   we have
$$
|\lambda^{-1}_{i}|=  o\left(i^{-\tau}    \right)\!,\,i\rightarrow \infty,\;0<\tau<\mu,
$$
Thus using the fact   $\lambda_{i}/i^{ \tau}\geq C,$   we can prove that  that there exists a subsequence $\{\lambda_{N_{\nu}}\}_{\nu=0}^{\infty},$
 such that
$$
|\lambda_{N_{\nu}+1}|-|\lambda_{N_{\nu}}|\geq K |\lambda_{N_{\nu}+1}|^{1-1/\tau},\,K>0,
$$
for this purpose it suffices to establish  the following implication
$$
\lim\limits_{n\rightarrow\infty}(\lambda_{n+1}-\lambda_{n})/\lambda^{(p-1)/p}_{n}=0,\;\Longrightarrow  \lim\limits_{n\rightarrow\infty} \lambda_{n}/n^{p}=0,\;p>0
$$
(see proof of Lemma 2 \cite{firstab_lit:2Agranovich1994}).
Now, consider
$$
|\lambda_{N_{\nu}+1}|-|\lambda_{N_{\nu}}|\geq K |\lambda_{N_{\nu}}|^{q},\;q:=1-1/\tau,
$$
and let us find $\delta_{\nu}$ from the condition $R_{\nu}=K|\lambda_{N_{\nu}}|^{q}+|\lambda_{N_{\nu}}|,\,R_{\nu}(1-\delta_{\nu})=|\lambda_{N_{\nu}}|,$ then $\delta_{\nu}^{-1}=1+K^{-1}|\lambda_{N_{\nu}}|^{1-q}.$ Further, we restrict our reasonings considering the   case $\mu\in(0,1],$ since the reasonings corresponding to the case $\alpha>1,\,\mu\in(1,\infty)$ is absolutely analogous.
Note that  in accordance with Lemma 3 \cite{firstab_lit:1Lidskii}  the following relation holds
$$
\sum\limits_{i=1}^{\infty}\lambda^{\frac{\varrho}{ (m+1)}}_{i}( \tilde{A} )\leq \sum\limits_{i=1}^{\infty}\lambda^{ \varrho }_{i}( A )<\infty,
$$
where   $\varrho$ is chosen so that   $2/\mu<\varrho<\alpha,\,m=[\varrho],\;\tilde{A}:=(A^{\ast m+1}A^{m+1})^{1/2}.$  It is clear that   $\tilde{A}\in \mathfrak{S}_{\varrho/(m+1)}.$
Consider  a  function
$$
\beta(r )= r^{ -\frac{\varrho}{ (m+1)} }\left(\int\limits_{0}^{r}\frac{n_{A^{m+1}}(t)dt}{t }+
r \int\limits_{r}^{\infty}\frac{n_{A^{m+1}}(t)dt}{t^{ 2  }}\right).
$$ Here, we produce the variant  of the proof corresponding to the case   $\mathrm{H}1.$   The variant of the proof corresponding to the case   $\Theta(A) \subset   \mathfrak{L}_{0}(\theta),\,\theta< \pi/2\alpha$ is   analogous and left to the reader. Consider a contour $\Gamma(A),$
absolutely analogously to the reasonings of Theorem \ref{T2}, applying Lemma \ref{L5}, we claim that  there exists an arch
$\tilde{\gamma}_{ \nu }:=\{\lambda:\;|\lambda|=\tilde{R}_{\nu},\,|\mathrm{arg } \lambda|< \theta_{\iota}+\varepsilon\} $    in the ring $(1-\delta_{\nu})R_{\nu}<|\lambda|<R_{\nu},$ on which the following estimate holds for  a sufficiently     small value $\delta>0$
$$
 I_{\nu} =\left\|\,\int\limits_{\tilde{\gamma}_{ \nu }}e^{-\lambda^{\alpha}t}A(I-\lambda A)^{-1}f d \lambda\,\right\|_{\mathfrak{H}}\leq C e^{|\lambda|^{\varrho}[\gamma (|\lambda|)-t|\lambda|^{\alpha-\varrho} \sin     \alpha \delta]}|\lambda|^{m+1}
 ,\,|\lambda|=\tilde{R}_{\nu} ,
$$
where
$\gamma (|\lambda|)= \beta  ( |\lambda|^{m+1})  +(2+\ln\{12e/\delta_{\nu}\}) \beta (|2e  \lambda| ^{m+1}) (2e) ^{\varrho}.$
It is clear that within the contour $\Gamma(A)$ between the arches $\tilde{\gamma}_{ \nu },\tilde{\gamma}_{ \nu+1 }$ (we denote the boundary of this  domain by $\gamma_{ \nu}$) there  lie the eigenvalues only for which the following relation holds
$$
|\lambda_{N_{\nu}+k}|-|\lambda_{N_{\nu}+k-1}|\leq C |\lambda_{N_{\nu}+k}|^{q},\;
 k=2,3,...,N_{\nu+1}-N_{\nu}.
$$
Using Lemma \ref{L8},   we   obtain a relation
$$
 \frac{1}{2\pi i}\int\limits_{\gamma_{ \nu}}e^{-\lambda^{\alpha}t}A(I-\lambda A)^{-1}f d \lambda
=     \sum\limits_{q=N_{\nu}+1}^{N_{\nu+1}}\sum\limits_{\xi=1}^{m(q)}\sum\limits_{i=0}^{k(q_{\xi})}e_{q_{\xi}+i}c_{q_{\xi}+i}(t).
$$
Hence,   to prove the main claim of the theorem, we should show that the series composed of    the above terms converges. Here, we want to realize  the idea of splitting $\gamma_{ \nu}$ on terms.
Let us prove that the   series $I$ converges.
Substituting $\delta_{\nu}^{-1},$ we have
$
\ln\{12e/\delta_{\nu}\} =   \ln\{12e+ 12eK^{-1}|\lambda_{N_{\nu}}|^{1-q }\}\leq C\ln\{ |\lambda_{N_{\nu}}|^{1-q }\}.
$
It is clear that to obtain the desired result we should prove that $$\ln  |\lambda_{N_{\nu}}|^{1-q} \beta (| \lambda_{N_{\nu}}| ^{m+1})\rightarrow 0,\;\nu\rightarrow\infty.$$
  Using simple reasonings based on the fact $\ln \!|\lambda_{N_{\nu}}| /|\lambda_{N_{\nu}}|^{\varepsilon} \rightarrow 0,\,\varepsilon>0,\,\nu\rightarrow\infty,$  applying Lemma \ref{L2}, we obtain the desired result.
  Finally, we should consider  the integrals along the contours $\gamma_{\nu_{+}}:=
\{\lambda:\,(1-\delta_{\nu})R_{\nu}<|\lambda|<R_{\nu},\, \mathrm{arg} \lambda  =\theta_{s} +\varepsilon\},\,\gamma_{\nu_{-}}:=
\{\lambda:\,(1-\delta_{\nu})R_{\nu}<|\lambda|<R_{\nu},\, \mathrm{arg} \lambda  =-\theta_{s}  -\varepsilon\},$ where $s=0,\iota.$     Analogously to Theorem \ref{T2}, applying Lemma \ref{L9}, we have
$$
 J^{+}_{\nu}: =\left\|\,\int\limits_{\gamma_{\nu_{+}}}e^{-\lambda^{\alpha}t}A(I-\lambda A)^{-1}f d \lambda\,\right\|_{\mathfrak{H}}\leq  C\|f\|_{\mathfrak{H}}    \int\limits_{(1-\delta_{\nu})R_{\nu} }^{R_{\nu} } |e^{- t \lambda^{\alpha}}|  |d   \lambda|\leq
   $$
   $$
   \leq C e^{-t(1-\delta_{\nu})^{\alpha}R_{\nu}  ^{\alpha} \sin     \alpha\delta} \!\!\!\!  \int\limits_{(1-\delta_{\nu})R_{\nu} }^{R_{\nu} }   |d   \lambda|=
  C e^{-t(1-\delta_{\nu})^{\alpha}R_{\nu}  ^{\alpha} \sin     \alpha\delta}  \delta_{\nu} R_{\nu};
$$
 $$
 J^{-}_{\nu}: =\left\|\,\int\limits_{\gamma_{\nu_{+}}}e^{-\lambda^{\alpha}t}A(I-\lambda A)^{-1}f d \lambda\,\right\|_{\mathfrak{H}} \leq C e^{-t(1-\delta_{\nu})^{\alpha}R_{\nu}  ^{\alpha} \sin     \alpha\delta}  \delta_{\nu} R_{\nu}.
$$
Therefore the series $J^{+},J^{-}$ are convergent. Thus, we obtain \eqref{13e}, from what follows the rest part of the theorem claim.
 \end{proof}
    Remaind that  a sequence of  contours of the exponential type   was  considered in the paper \cite{firstab_lit:1Lidskii}, under   the imposed  condition $\alpha>\rho.$  We improve this result in the following  sense,   we produce a sequence of  contours of the power type,  what  gives us a solution of the problem in the case $A\in \tilde{\mathfrak{S}}_{\rho},\;\alpha=\rho.$
\begin{teo}\label{T4}
Assume that a normal operator $\tilde{W}$ satisfies the condition $
 (\ln^{1+1/\alpha}x)'_{\lambda_{i}(H)}  =o(  i^{-1/\alpha}   ),$
  $\alpha>1.$
Then a sequence of the natural  numbers  $\{N_{\nu}\}_{0}^{\infty}$ can be chosen  so that
\begin{equation*}\label{2a}
 \frac{1}{2\pi i}\int\limits_{\varpi}e^{-\lambda^{\alpha}t}A(I-\lambda A)^{-1}f d \lambda
=  \sum\limits_{\nu=0}^{\infty}  \sum\limits_{q=N_{\nu}+1}^{N_{\nu+1}}\sum\limits_{\xi=1}^{m(q)}\sum\limits_{i=0}^{k(q_{\xi})}e_{q_{\xi}+i}c_{q_{\xi}+i}(t),
\end{equation*}
moreover
\begin{equation}\label{13b1}
  \sum\limits_{\nu=0}^{\infty}\left\|\sum\limits_{q=N_{\nu}+1}^{N_{\nu+1}}\sum\limits_{\xi=1}^{m(q)}\sum\limits_{i=0}^{k(q_{\xi})}e_{q_{\xi}+i}c_{q_{\xi}+i}(t)\right\|_{\mathfrak{H}}<\infty,
\end{equation}
 the following relation holds for the corresponding  eigenvalues
$$
|\lambda_{N_{\nu}+k}|-|\lambda_{N_{\nu}+k-1}|\leq C |\lambda_{N_{\nu}+k}|^{1-1/\tau},\;
 k=2,3,...,N_{\nu+1}-N_{\nu},\;0<\tau<1/\alpha.
$$
\end{teo}
\begin{proof}
Applying Theorem 1, we get
$$
|\lambda^{-1}_{i}|=  o\left(i^{-\tau}    \right)\!,\,i\rightarrow \infty,\;0<\tau<1/\alpha,
$$
Thus, using the fact   $\lambda_{i}/i^{ \tau}\geq C,$  we can prove that  that there exists a subsequence $\{\lambda_{N_{\nu}}\}_{\nu=0}^{\infty},$
 such that
$$
|\lambda_{N_{\nu}+1}|-|\lambda_{N_{\nu}}|\geq K |\lambda_{N_{\nu}+1}|^{1-1/\tau},\,K>0,
$$
for this purpose it suffices to establish  the following implication
$$
\lim\limits_{n\rightarrow\infty}(\lambda_{n+1}-\lambda_{n})/\lambda^{(p-1)/p}_{n}=0,\;\Longrightarrow  \lim\limits_{n\rightarrow\infty} \lambda_{n}/n^{p}=0,\;p>0
$$
(see proof of Lemma 2 \cite{firstab_lit:2Agranovich1994}).
Now, consider
$$
|\lambda_{N_{\nu}+1}|-|\lambda_{N_{\nu}}|\geq K |\lambda_{N_{\nu}}|^{q},\;q:=1-1/\tau,
$$
and let us find $\delta_{\nu}$ from the condition $R_{\nu}=K|\lambda_{N_{\nu}}|^{q}+|\lambda_{N_{\nu}}|,\,R_{\nu}(1-\delta_{\nu})=|\lambda_{N_{\nu}}|,$ then $\delta_{\nu}^{-1}=1+K^{-1}|\lambda_{N_{\nu}}|^{1-q}.$ In accordance with Lemma \ref{L4}, we have   $A\in \tilde{\mathfrak{S}}_{\rho} ,\;\rho\in[0,\alpha],\; n_{A}(r)=o( r^{\alpha}/\ln r).$
Here, we produce the variant  of the proof corresponding to the case   $\mathrm{H}1.$   The variant of the proof corresponding to the case   $\Theta(A) \subset   \mathfrak{L}_{0}(\theta),\,\theta< \pi/2\alpha$ is absolutely  analogous and left to the reader. Consider a contour $\Gamma(A),$
 applying Lemma \ref{L5}  analogously to the reasonings of Theorem \ref{T2},  we claim that  for a sufficiently   small $\delta>0,$   there exists an arch
$\tilde{\gamma}_{ \nu }:=\{\lambda:\;|\lambda|=\tilde{R}_{\nu},\,|\mathrm{arg } \lambda|< \theta +\varepsilon\}$    in the ring $(1-\delta_{\nu})R_{\nu}<|\lambda|<R_{\nu},$ on which the following estimate holds
\begin{equation}\label{13b}
 I_{\nu} =\left\|\,\int\limits_{\tilde{\gamma}_{ \nu }}e^{-\lambda^{\alpha}t}A(I-\lambda A)^{-1}f d \lambda\,\right\|_{\mathfrak{H}}\leq e^{|\lambda|^{\alpha}[\gamma (|\lambda|)-t  \sin     \alpha \delta]}|\lambda|^{m+1}
 ,\,m=[\alpha],\,|\lambda|=\tilde{R}_{\nu},
\end{equation}
where
$\gamma(|\lambda|)= \beta ( |\lambda|^{m+1})  +(2+\ln\{12e/\delta_{\nu}\}) \beta(|2e  \lambda| ^{m+1}) (2e) ^{\alpha},$
$$
\beta(r )= r^{ -\frac{\alpha}{  m+1 } }\left(\int\limits_{0}^{r}\frac{n_{A^{m+1}}(t)dt}{t }+
r \int\limits_{r}^{\infty}\frac{n_{A^{m+1}}(t)dt}{t^{ 2  }}\right).
$$
It is clear that within the contour $\Gamma(A),$ between the arches $\tilde{\gamma}_{ \nu },\tilde{\gamma}_{ \nu+1 }$ (we denote the boundary of this  domain by $\gamma_{ \nu}$) there  lie the eigenvalues only for which the following relation holds
$$
|\lambda_{N_{\nu}+k}|-|\lambda_{N_{\nu}+k-1}|\leq C |\lambda_{N_{\nu}+k}|^{q},\;
 k=2,3,...,N_{\nu+1}-N_{\nu}.
$$
Using Lemma \ref{L8},   we   obtain a relation
$$
 \frac{1}{2\pi i}\int\limits_{\gamma_{ \nu}}e^{-\lambda^{\alpha}t}A(I-\lambda A)^{-1}f d \lambda
=     \sum\limits_{q=N_{\nu}+1}^{N_{\nu+1}}\sum\limits_{\xi=1}^{m(q)}\sum\limits_{i=0}^{k(q_{\xi})}e_{q_{\xi}+i}c_{q_{\xi}+i}(t).
$$
It is clear that to obtain the desired result, we should prove that the series composed  of the above terms converges. However, we can prove   that the series $I$ converges, what is more stronger condition.  Here, we want to realize  the idea of splitting $\gamma_{ \nu}$ on terms.
 Consider the right-hand side of formula \eqref{13b}.
Substituting $\delta^{-1},$ we have
$$
\ln\{12e/\delta\} =   \ln\{12e+ 12eK^{-1}|\lambda_{N_{\nu}}|^{1-q }\}\leq C\ln\{ |\lambda_{N_{\nu}}|^{1-q }\}.
$$
Hence, to obtain the desired result we should prove that $\ln  |\lambda_{N_{\nu}}|^{1-q} \beta (| \lambda_{N_{\nu}}| ^{m+1})\rightarrow 0,\,\nu\rightarrow\infty.$ In its own turn, using  Lemma \ref{L3}, we can prove the latter relation, if we show that
\begin{equation}\label{13a}
\ln r\frac{n_{A^{m+1}}(r^{m+1})}{r^{\alpha} }\rightarrow 0,\;r\rightarrow\infty.
\end{equation}
We need  establish some facts. Notice  that the following operators  have the same eigenfunctions  i.e.
\begin{equation}\label{14a}
(A^{\ast}A )^{1/2}f_{n}=\mu_{n}f_{n} \Longleftrightarrow (A^{ s \ast}A^{s})^{1/2}f_{n}=\mu^{s}_{n}f_{n},\,s=m+1.
\end{equation}
To prove this fact, firstly let us show that $A^{\ast s}=A^{s \ast},$ it follows easily from the inclusion $A^{\ast s}\subset A^{s \ast}$ and the fact
$\mathrm{D}(A^{\ast s})=\mathfrak{H}.$ Thus, for a normal operator we have $(A^{\ast}A)^{s}=A^{\ast  s}A^{s}.$ Let us involve the spectral theorem for    the selfadjoint non-negative  operator, in accordance with a definition (see \cite{firstab_lit:Krasnoselskii M.A.} Chapter 3), we have
$$
(A^{\ast}A)^{\vartheta}=\int\limits_{0}^{\|A^{\ast}A\|}\!\lambda^{\vartheta}dP_{\lambda},
\,\vartheta>0,
$$
where the latter integral is  understood in the Riemann  sense as a limit of the partial sums
\begin{equation*}
\sum\limits_{i=0}^{n}\xi^{\vartheta}_{i} P_{\Delta \lambda_{i}}  \stackrel{\mathfrak{H}}{ \longrightarrow} \int\limits_{0}^{\|A^{\ast}A\|}\lambda^{\vartheta}dP_{\lambda},\,\omega\rightarrow 0,
\end{equation*}
where $(0=\lambda_{0}<\lambda_{1}<...<\lambda_{n}=\|A^{\ast}A\|)$ is an arbitrary splitting of the segment $[0,\|A^{\ast}A\|],\;\omega:=\max\limits_{i}(\lambda_{i+1}-\lambda_{i}),\;\xi_{i}$ is an arbitrary point belonging to $[\lambda_{i},\lambda_{i+1}],$ the operators $P_{\Delta \lambda_{i}}$ are projectors corresponding to the selfadjoint operator. It follows easily from the well-known facts that if in additional $A^{\ast}A$ is a compact operator, then the above formula reduces to
\begin{equation*}
(A^{\ast}A)^{\vartheta}f=\sum\limits_{n=1}^{\infty} \lambda^{\vartheta}_{n}(f,\varphi_{n})_{\mathfrak{H}}\varphi_{n},
\end{equation*}
where $\{\varphi_{n}\}_{1}^{\infty},\,\{\lambda_{n}\}_{1}^{\infty}$ are  sets of the  eigenvectors and the eigenvalues of the operator $A^{\ast}A$ respectively. Taking into account   the latter representation, an obvious fact that $(A^{\ast}A)^{\vartheta}$ is selfadjoint, it is not hard to prove
$$
(A^{\ast}A)^{ \frac{1}{2}\cdot s}=(A^{\ast}A)^{  s  \cdot \frac{1}{2}}.
$$
Thus, using the property $A^{\ast}A=AA^{\ast},$ we get
$$
(A^{\ast}A)^{ \frac{1}{2}\cdot s}=(A^{s\ast}A^{s})^{\frac{1}{2}},
$$
from what follows the implication from the left-hand side of formula \eqref{14a}. To obtain the  contrary  implication we should establish the fact that  the operator and its positive powers  have the same  eigenvectors. For this purpose, let us notice that
$$
T^{\vartheta}\varphi_{i}=\lambda^{\vartheta}_{i}\varphi_{i},\,i\in \mathbb{N},
$$
where $T:=A^{\ast}A.$ It follows that
$$
T^{\vartheta}f=\sum\limits_{n=1}^{\infty} \lambda^{\vartheta}_{n}(f,\varphi_{n})_{\mathfrak{H}}\varphi_{n}=\sum\limits_{n=1}^{\infty}  (f,T^{\vartheta}\varphi_{n})_{\mathfrak{H}}\varphi_{n}=\sum\limits_{n=1}^{\infty}  (T^{\vartheta}f, \varphi_{n})_{\mathfrak{H}}\varphi_{n}.
$$
 Hence, we have a fact
 $$
g =\sum\limits_{n=1}^{\infty}  (g, \varphi_{n})_{\mathfrak{H}}\varphi_{n},\;g\in \mathrm{R}(T^{\vartheta}).
 $$
 Let us assume that there exists an eigenfunction $h$ of the operator $T^{\vartheta}$ that differs from $\varphi_{i},\;i\in \mathbb{N}.$ Using the  proved above fact,  we get
$$
T^{\vartheta}h=\sum\limits_{n=1}^{\infty} \lambda^{\vartheta}_{n}(h,\varphi_{n})_{\mathfrak{H}}\varphi_{n}=\zeta\sum\limits_{n=1}^{\infty}  ( h, \varphi_{n})_{\mathfrak{H}}\varphi_{n},
$$
where $\zeta$  is a corresponding eigenvalue. Multiplying (in the sense of the inner product) both sides of the latter relation on $\varphi_{k},\varphi_{k+1}$ we get $\lambda^{\vartheta}_{k}= \zeta=\lambda^{\vartheta}_{k+1},$   this contradiction proves the desired result.
Thus, we complete the proof of formula \eqref{14a}. To complete the proof of relation \eqref{13a} we need   mention  the fact $n_{A}(\lambda)=n_{A^{m}}(\lambda^{m})$ that follows easily from relation \eqref{14a}. Thus, making a substitution and using  the theorem condition, we claim that relation \eqref{13a} holds, hence the series $I$ is convergent. To complete the proof, we should note that the integrals along the following  contours converges uniformly  $\gamma_{\nu_{+}}:=
\{\lambda:\,(1-\delta_{\nu})R_{\nu}<|\lambda|<R_{\nu},\, \mathrm{arg} \lambda  =\theta_{s}  +\varepsilon\},\,\gamma_{\nu_{-}}:=
\{\lambda:\,(1-\delta_{\nu})R_{\nu}<|\lambda|<R_{\nu},\, \mathrm{arg} \lambda  =-\theta_{s}  -\varepsilon\},$ where $s=0,\iota.$    Analogously to the reasonings of  Theorem \ref{T2}, applying Lemma \ref{L9}, we have
$$
 J^{+}_{\nu}: =\left\|\,\int\limits_{\gamma_{\nu_{+}}}e^{-\lambda^{\alpha}t}A(I-\lambda A)^{-1}f d \lambda\,\right\|_{\mathfrak{H}}\leq C\|f\|_{\mathfrak{H}} \!\!\! \int\limits_{(1-\delta_{\nu})R_{\nu} }^{R_{\nu} } |e^{- t \lambda^{\alpha}}|  |d   \lambda|\leq C e^{-t(1-\delta)^{\alpha}R_{\nu}  ^{\alpha} \sin     \alpha\delta}   \!\!\!\int\limits_{(1-\delta_{\nu})R_{\nu} }^{R_{\nu} }   |d   \lambda|=
 $$
 $$
 = Ce^{-t(1-\delta_{\nu})^{\alpha}R_{\nu}  ^{\alpha} \sin     \alpha\,\delta}  \delta_{\nu} R_{\nu};
$$
 $$
J^{-}_{\nu}: =\left\|\,\int\limits_{\gamma_{\nu_{-}}}e^{-\lambda^{\alpha}t}A(I-\lambda A)^{-1}f d \lambda\,\right\|_{\mathfrak{H}}\leq Ce^{-t(1-\delta_{\nu})^{\alpha}R_{\nu}  ^{\alpha} \sin     \alpha\delta}  \delta_{\nu} R_{\nu}.
$$
Therefore, the   series $J^{+},J^{-}$ converge. Thus, we obtain relation \eqref{13b1}, from what follows the rest part of the proof.
\end{proof}
\begin{corol}\label{C1} Under  the Theorem \ref{T4} assumptions, we get
$$
f
=  \lim\limits_{t\rightarrow+0}\sum\limits_{\nu=0}^{\infty}  \sum\limits_{q=N_{\nu}+1}^{N_{\nu+1}}\sum\limits_{\xi=1}^{m(q)}\sum\limits_{i=0}^{k(q_{\xi})}e_{q_{\xi}+i}c_{q_{\xi}+i}(t),\;f\in \mathrm{D}(\tilde{W}).
$$
This fact follows  immediately from  Lemmas \ref{L7},\ref{L10} respectively.
\end{corol}

\noindent{\bf Differential equations in the Hilbert space}\\

 Further, we will consider a Hilbert space $\mathfrak{H}$ consists of   element-functions $u:\mathbb{R}_{+}\rightarrow \mathfrak{H},\,u:=u(t),\,t\geq0$    and we will assume that if $u$ belongs to $\mathfrak{H}$    then the fact  holds for all values of the variable $t.$ Notice that under such an assumption all standard topological  properties as completeness, compactness e.t.c. remain correctly defined. We understand such operations as differentiation and integration in the generalized sense that is caused by the topology of the Hilbert space $\mathfrak{H}.$ The derivative is understood as the following  limit
$$
  \frac{u(t+\Delta t)-u(t)}{\Delta t}\stackrel{\mathfrak{H}}{ \longrightarrow}\frac{du}{dt} ,\,\Delta t\rightarrow 0.
$$
Let $t\in I:=[a,b],\,0< a <b<\infty.$ The following integral is understood in the Riemann  sense as a limit of partial sums
\begin{equation*}
\sum\limits_{i=0}^{n}u(\xi_{i})\Delta t_{i}  \stackrel{\mathfrak{H}}{ \longrightarrow}  \int\limits_{I}u(t)dt,\,\lambda\rightarrow 0,
\end{equation*}
where $(a=t_{0}<t_{1}<...<t_{n}=b)$ is an arbitrary splitting of the segment $I,\;\lambda:=\max\limits_{i}(t_{i+1}-t_{i}),\;\xi_{i}$ is an arbitrary point belonging to $[t_{i},t_{i+1}].$
The sufficient condition of the last integral existence is a continuous property (see\cite[p.248]{firstab_lit:Krasnoselskii M.A.}) i.e.
$
u(t)\stackrel{\mathfrak{H}}{ \longrightarrow}u(t_{0}),\,t\rightarrow t_{0},\;\forall t_{0}\in I.
$
The improper integral is understood as a limit
\begin{equation*}
 \int\limits_{a}^{b}u(t)dt\stackrel{\mathfrak{H}}{ \longrightarrow} \int\limits_{a}^{c}u(t)dt,\,b\rightarrow c,\,c\in [0,\infty].
\end{equation*}
Combining these operations we can consider a  fractional differential operator
 in the Riemann-Liouvile sense (see \cite{firstab_lit:samko1987})   i.e. in the formal form, we have
$$
   \mathfrak{D}^{1/\alpha}_{-}f(t):=-\frac{1}{\Gamma(1-1/\alpha)}\frac{d}{d t}\int\limits_{0}^{\infty}f(t+x)x^{-1/\alpha}dx,\;\alpha>1.
$$
Let us study   a Cauchy problem
\begin{equation}\label{17n}
   \mathfrak{D}^{1/\alpha}_{-}  u=\tilde{W}u ,\;u(0)=h\in \mathrm{D}(\tilde{W}),
\end{equation}
in the case  when the operator composition $\mathfrak{D}^{1-1/\alpha}_{-}\tilde{W}$ is   accretive  we assume  that   $h\in \mathfrak{H}.$
 \begin{teo}\label{T5}
Assume that  the Theorem \ref{T4} conditions  hold,   then
there exists a solution of the Cauchy problem \eqref{17n} in the form
\begin{equation}\label{23n}
u(t)= \frac{1}{2\pi i}\int\limits_{\varpi}e^{-\lambda^{\alpha}t}A(I-\lambda A)^{-1}h d \lambda
=  \sum\limits_{\nu=0}^{\infty}  \sum\limits_{q=N_{\nu}+1}^{N_{\nu+1}}\sum\limits_{\xi=1}^{m(q)}\sum\limits_{i=0}^{k(q_{\xi})}e_{q_{\xi}+i}c_{q_{\xi}+i}(t),
\end{equation}
where
\begin{equation*}
  \sum\limits_{\nu=0}^{\infty}\left\|\sum\limits_{q=N_{\nu}+1}^{N_{\nu+1}}\sum\limits_{\xi=1}^{m(q)}\sum\limits_{i=0}^{k(q_{\xi})}e_{q_{\xi}+i}c_{q_{\xi}+i}(t)\right\|<\infty,
\end{equation*}
a sequence of natural numbers $\{N_{\nu}\}_{0}^{\infty}$ can be chosen in accordance with the claim of    Theorem \ref{T4}.
Moreover, the existing solution is unique if the operator composition  $\mathfrak{D}^{1-1/\alpha}_{-}\tilde{W}$ is accretive.
\end{teo}
\begin{proof}
  Let us find a solution in the form \eqref{23n} satisfying the initial condition \eqref{17n}. Bellow, we produce the variant  of the proof corresponding to the case  $\Theta(A) \subset   \mathfrak{L}_{0}(\theta),\,\theta< \pi/2\alpha.$    The variant of the proof corresponding to the case $\mathrm{H}1$   is absolutely  analogous and left to the reader.
   Consider a contour $\gamma(A).$ 
Using Lemma \ref{L6}, it is not hard to prove  that the following  integral exists i.e.
$$
 \frac{1}{2\pi i} \int\limits_{\gamma(A)}e^{-\lambda^{\alpha}t} (E-\lambda A)^{-1}h d\lambda  \in \mathfrak{H};\;\frac{d u}{d t}=-\frac{1}{2\pi i} \int\limits_{\gamma(A)}e^{-\lambda^{\alpha}t}  \lambda^{\alpha } A(E-\lambda A)^{-1}h \,d\lambda \in \mathfrak{H}.
 $$
Note that the first relation gives us the fact
   $u(t)\in \mathrm{D}(\tilde{W}).$
 Using Lemmas \ref{L5},\ref{L6}  analogously to the methods of the ordinary calculus, we can establish the following formulas
$$
\int\limits_{0}^{\infty}x^{-1/\alpha}dx\int\limits_{\gamma(A)}e^{-\lambda^{\alpha}(t+x)}A(E-\lambda A)^{-1}h d\lambda=\int\limits_{\gamma(A)}e^{-\lambda^{\alpha}t}A(E-\lambda A)^{-1}h d\lambda\int\limits_{0}^{\infty}x^{-1/\alpha}e^{-\lambda^{\alpha}x}dx;
$$
$$
\frac{d}{dt}\int\limits_{\gamma(A)}\lambda^{1-\alpha}e^{-\lambda^{\alpha}t}A(E-\lambda A)^{-1}h d\lambda =
-\int\limits_{\gamma(A)}\lambda e^{-\lambda^{\alpha}t}A(E-\lambda A)^{-1}h d\lambda.
$$
Therefore, combining these formulas, taking into account a relation
$$
  \int\limits_{0}^{\infty}x^{-1/\alpha}e^{-\lambda^{\alpha}x}dx=  \Gamma(1-1/\alpha) \lambda^{1-\alpha},
$$
we get
$$
\mathfrak{D}^{1/\alpha}_{-}u=\frac{1}{2\pi i} \int\limits_{\gamma(A)}e^{-\lambda^{\alpha}t}  \lambda  A(E-\lambda A)^{-1}h d\lambda.
$$
Making a substitution using   the formula $\lambda A(E-\lambda A)^{-1}=(E-\lambda A)^{-1}-E,$   we obtain
$$
  \mathfrak{D}^{1/\alpha}_{-}u= \frac{1}{2\pi i} \int\limits_{\gamma(A)}e^{-\lambda^{\alpha}t}   (E-\lambda A)^{-1}h\, d\lambda-\frac{1}{2\pi i} \int\limits_{\gamma(A)}e^{-\lambda^{\alpha}t}   h \,d\lambda=I_{1}+I_{2}.
$$
The second integral equals zero by virtue  of  the fact  that the   function under the integral is analytical inside the intersection of the  domain $\mathrm{int}\, \gamma(A)$  with the circle of  an arbitrary radius $R$   and it   decreases sufficiently fast on the arch of the radius $R,$ when $R\rightarrow\infty.$ Now, if we consider the    expression for $u,$ we obtain the fact that $u$ is a solution of the equation i.e.
$
\mathfrak{D}^{1/\alpha}_{-}u=\tilde{W }u.
$
 The decomposition   on  the series  of the root vectors \eqref{23n} is obtained due to  Theorem   \ref{T4}.
 Let us show that the initial condition holds in the sense
$
u(t)   \xrightarrow[   ]{\mathfrak{H}}  h,\,t\rightarrow+0.
$
It becomes clear in the case  $h\in \mathrm{D}(\tilde{W}),$   it suffices to apply Lemma \ref{L7},     what  gives  us the desired result. Consider a case when $h$ is an arbitrary element of the Hilbert space $\mathfrak{H}.$ Let us involve the accretive property of the operator composition  $ \mathfrak{D}^{1-1/\alpha}_{-}\tilde{W}  .$
It follows from Lemma \ref{L6}   that for a fixed $t$ the  following operator is bounded
$$
S_{t}h=\frac{1}{2\pi i} \int\limits_{\gamma(A)}e^{-\lambda^{\alpha}t}A(E-\lambda A)^{-1}h d\lambda.
$$
  Let us show that
$
\|S_{t}\|\leq1,\;t>0.
$
Firstly, assume that $h\in \mathrm{D}(\tilde{W}),$ then in accordance with the above,  we get
$
u(t)   \xrightarrow[   ]{\mathfrak{H}}  h,\,t\rightarrow+0.
$
Thus, we can claim the fact that   $u(t) $ is continuous at the right-hand side of the point zero.  Let us apply the operator $\mathfrak{D}^{1-1/\alpha}_{-}$ to the both sides of relation \eqref{17n}. Taking into account a relation
$
\mathfrak{D}^{1-1/\alpha}_{-}\mathfrak{D}^{1/\alpha}_{-}u=-u',
$
we get
$
u'+\mathfrak{D}^{1-1/\alpha}_{-}\tilde{W}u=0.
$
 Let us multiply the both sides of the latter  relation  on $u$  in the sense of the inner product, we get
$
\left(u'_{t},u\right)_{\mathfrak{H}}+(\mathfrak{D}^{1-1/\alpha}_{-}\tilde{W}u,u)_{\mathfrak{H}}=0.
$
Consider a real part of the latter relation, we have
$
\mathrm{Re}\left(u'_{t},u\right)_{\mathfrak{H}}+\mathrm{Re}(\mathfrak{D}^{1-1/\alpha}_{-}\tilde{W}u,u)_{\mathfrak{H}}= \left(u'_{t},u\right)_{\mathfrak{H}}/2+ \left(u, u'_{t}\right)_{\mathfrak{H}}/2+\mathrm{Re}(\mathfrak{D}^{1-1/\alpha}_{-}\tilde{W}u,u)_{\mathfrak{H}}.
$
Therefore
$
   \left(\|u(t)\|_{\mathfrak{H}}^{2}\right)'_{t}  =-2\mathrm{Re}(\mathfrak{D}^{1-1/\alpha}_{-}\tilde{W}u,u)_{\mathfrak{H}}\leq 0.
$
Integrating both sides, we get
$$
  \|u(\tau)\|_{\mathfrak{H}}^{2}-  \|u(0)\|_{\mathfrak{H}}^{2}=\int\limits_{0}^{\tau} \frac{d }{dt}\|u(t)\|_{\mathfrak{H}}^{2} dt\leq 0.
$$
The last relation can be rewritten in the form
$
\|S_{t}h\|_{\mathfrak{H}}\leq \|h\|_{\mathfrak{H}},\,h\in \mathrm{D}(\tilde{W}).
$
Since $\mathrm{D}(\tilde{W})$ is a dense set in $ \mathfrak{H},$ then we obviously  obtain  the  desired result i.e. $\|S_{t}\|\leq 1.$  Now consider the following reasonings, having assumed that
$
h_{n}   \xrightarrow[   ]{\mathfrak{H}}  h,\,n\rightarrow \infty,\;\{h_{n}\}\subset \mathrm{D}(\tilde{W}),\,h\in \mathfrak{H},
$
we have
$
\|u(t)-h\|_{\mathfrak{H}}=\|S_{t}h-h\|_{\mathfrak{H}}=\|S_{t}h-S_{t}h_{n}+S_{t}h_{n}-h_{n}+h_{n}-    h\|_{\mathfrak{H}}\leq \|S_{t}\|\cdot\|h- h_{n}\|_{\mathfrak{H}}+\|S_{t}h_{n}-h_{n}\|_{\mathfrak{H}}+\|h_{n}-    h\|_{\mathfrak{H}}.
$
It is clear that if we chose $n$ so that  $\|h- h_{n}\|_{\mathfrak{H}}<\varepsilon/3$ and  after that  chose $t$   so that $\|S_{t}h_{n}-h_{n}\|_{\mathfrak{H}}<\varepsilon,$ then we obtain
 $\forall\varepsilon>0,\,\exists \delta(\varepsilon):\,\|u(t)-h\|_{\mathfrak{H}}<\varepsilon,\,t<\delta.$ Thus the initial condition holds.   The uniqueness follows easily from the fact that $\mathfrak{D}^{1-1/\alpha}_{-}\tilde{W}$ is accretive. In this case, repeating the previous reasonings we come to
\begin{equation}\label{24m}
  \|g(\tau)\|^{2}-  \|g(0)\|_{\mathfrak{H}}^{2}=\int\limits_{0}^{\tau} \frac{d }{dt}\|g(t)\|^{2} dt\leq 0,
\end{equation}
where $g $ is a sum of two solutions $u_{1} $ and $u_{2}.$ Notice that  by virtue of the initial conditions, we have  $g(0)=0,$  thus    relation \eqref{24m} can hold only if $g =0.$  The proof is complete.
\end{proof}

\end{document}